\documentclass[12pt, reqno]{smfart}
\usepackage{amsmath,amssymb,amsthm,amsfonts,amscd, url, leftidx,bbold,}
\usepackage[mathscr]{eucal}
\usepackage{hyperref}
\usepackage[numeric]{amsrefs}
\usepackage{bm}

\usepackage{enumerate}

\input xy
\xyoption{all}

\newcounter{Todo}

\usepackage{mathptmx} 
\usepackage[scaled=0.90]{helvet} 
\usepackage{courier} 
\normalfont
\usepackage[T1]{fontenc}

\usepackage[dvipsnames]{xcolor}
\usepackage{graphics}
\usepackage{graphicx}

\usepackage{tikz,tikz-cd}
\usetikzlibrary{shapes.geometric, arrows}
\tikzstyle{block} = [rectangle, rounded corners, minimum width=3cm, minimum height=1cm,text centered, text width=3cm, draw=black]
\usepackage[edge length=0.8cm]{dynkin-diagrams}

\newcommand{\tpoint}[1]{\subsubsection{#1}}
\newcommand{\spoint}{\subsubsection{}}

\newtheorem*{nthm}{Theorem}
\newtheorem*{nlem}{Lemma}
\newtheorem*{nprop}{Proposition}
\newtheorem*{ncor}{Corollary}

\newtheorem*{nclaim}{Claim}
\newtheorem*{nexam}{Example}
\newtheorem*{nqe}{Question}

\theoremstyle{definition}

\newtheorem*{nrem}{Remarks}
\theoremstyle{remark}

\newcommand{\be}{\begin{eqnarray}}
\newcommand{\ee}{\end{eqnarray}}

\numberwithin{equation}{section}

\newcommand{\mf}[1]{\mathfrak{#1}}

\newcommand{\mc}[1]{\mathcal{{#1}}}
\renewcommand{\mod}{\mathrm{mod}}
\newcommand{\rr}{\rightarrow}
\newcommand{\la}{\langle}
\newcommand{\ra}{\rangle}

\DeclareMathOperator{\fin}{fin}
\DeclareMathOperator{\res}{res}

\DeclareMathOperator{\GL}{GL}

\newcommand{\Aut}{\operatorname{Aut}}

\newcommand{\alphav}{{\alpha}^\vee}
\newcommand{\betav}{\check{\betav}}

\newcommand{\lv}{{\lambda}^\vee}
\newcommand{\mv}{\check{\mv}}

\newcommand{\R}{{\mathbb{R}}}
\newcommand{\C}{{\mathbb{C}}}
\newcommand{\Q}{{\mathbb{Q}}}
\newcommand{\ad}{\mathbb{A}}

\renewcommand{\O}{{\mathcal{O}}}
\renewcommand{\o}{\mathfrak{o}}

\newcommand{\finf}{\mathbb{F}}
\newcommand{\leg}[2]{ \left( \frac{#1}{#2} \right) }

\newcommand{\bG}{\mathbf{G}}

\newcommand{\bB}{\mathbf{B}}

\newcommand{\bA}{\mathbf{A}}

\newcommand{\bU}{\mathbf{U}}
\newcommand{\bT}{\mathbf{T}}

\newcommand{\Lv}{\Lambda^\vee}

\newcommand{\dw}{\dot{w}}
\newcommand{\tT}{\widetilde{T}}
\newcommand{\tG}{\widetilde{G}}

 \newcommand{\BA}{{\mathbb {A}}} 
    \newcommand{\BC}{{\mathbb {C}}} 
     \newcommand{\BF}{{\mathbb {F}}}

     \newcommand{\BN}{{\mathbb {N}}}
     
    \newcommand{\BQ}{{\mathbb {Q}}} \newcommand{\BR}{{\mathbb {R}}}

     \newcommand{\BZ}{{\mathbb {Z}}}
    \newcommand{\Bg}{{\mathbb g}}

    \newcommand{\CE}{{\mathcal {E}}}

    \newcommand{\CK}{{\mathcal {K}}} 
     
    \newcommand{\CO}{{\mathcal {O}}} 
     
     \newcommand{\CT}{{\mathcal {T}}}
     
    \newcommand{\CW}{{\mathcal {W}}}

    \newcommand{\fo}{{\mathfrak{o}}} \newcommand{\fp}{{\mathfrak{p}}}

     \newcommand{\fD}{{\mathfrak{D}}}

    \newcommand{\id}{{\mathrm{id}}}
    \newcommand{\Ind}{{\mathrm{Ind}}}

    \newcommand{\ord}{{\mathrm{ord}}} \newcommand{\rank}{{\mathrm{rank}}}

    \renewcommand{\mod}{\ \mathrm{mod}\ }\renewcommand{\Re}{{\mathrm{Re}}}

 	\newcommand{\supp}{{\mathrm{supp}}}

    \newcommand{\SL}{{\mathrm{SL}}}
    \newcommand{\Spec}{{\mathrm{Spec}}}

    \newcommand{\univ}{{\mathrm{univ}}}

	\newcommand{\ul}{\underline}

    \newcommand{\pair}[1]{\langle {#1} \rangle}

    \newcommand{\bs}{\bm{i}}

    \newcommand{\adeles}{ad\`{e}les~}

    \newcommand{\Wh}{{\mathrm{Wh}}}

\newcommand{\se}{{\mathsf e}}
\newcommand{\ro}{\mathcal{O}}
\newcommand{\Gode}{\mathrm{Gode}}
\newcommand{\sQ}{{\mathsf Q}}
\newcommand{\cs}{\widetilde{\mathbf{CS}}}
\newcommand{\sB}{{\mathsf B}}
\newcommand{\lf}{\mathscr{F}}

\begin{document}

\title{Eisenstein Series on Metaplectic Covers and Multiple Dirichlet Series}

\author{ Yanze Chen }
\address{Department of Mathematical and Statistical Sciences, University of Alberta, Edmonton, AB
T6G 2G1, Canada}
\email{yanze4@ualberta.ca}
\maketitle 
\begin{abstract}
	We computed the first Whittaker coefficient of an Eisenstein series on a global metaplectic group induced from the torus and related the result with a Weyl group multiple Dirichlet series attached to the (dual) root system of the group under a mild assumption on the root system and the degree of the metaplectic cover. This confirms a conjecture of Brubaker-Bump-Friedberg.
\end{abstract}
\setcounter{tocdepth}{2}
\tableofcontents

\section{Introduction}
\subsection{Statement of main result}
Let $k$ be a global field with ring of \adeles $\BA$, $\fD=(\Lambda,\Delta,\Lambda^\vee,\Delta^\vee)$ be a semisimple simply-connected based root datum, $n$ be a positive integer such that $k$ contains all $2n$-th roots of unity, let $S$ be a finite set of places of $k$ containing all the archimedean places and satisfies some additional conditions (see \S\ref{subsub:WMDS-S-conditions}). There are two objects that are naturally associated to the above data $(k,\fD,n,S)$:
\begin{enumerate}[(1)]
	\item the \emph{metaplectic Eisenstein series} $E(g,\Phi_{\lambda, S},\lambda)$, which is an Eisenstein series on the global metaplectic group $\widetilde G_\BA$ with root datum $\fD$ induced from the torus, defined in  \S\ref{met-Es}.	\item the \emph{Weyl group multiple Dirichlet series} (WMDS) $Z_\psi(s_1,\cdots,s_r)$, which is a Dirichlet series in multiple variables with twisted multiplicative coefficients, defined in \S\ref{WMDS}. 
\end{enumerate}

The goal of this article is to prove the following
\begin{nthm}\label{main-result}
	Suppose the metaplectic dual root datum $\fD_{(\sQ,n)}^\vee$ is of adjoint type (see \S\ref{notation: root system}). Let $\psi$ be an additive character of $\BA/k$ that is unramified over every place $\nu\notin S$. For $\lambda\in\Gode$ in the Godement region (\ref{Godement-region}), the first Whittaker coefficient of $E(g,\lambda)$
	$$\int_{U_\BA^-/U_k^-}E(\lambda,\Phi_{\lambda,S},u^-)\psi(u^-)^{-1}du$$
	is equal to 
	$$[T_{\fo_S}:T_{0,\fo_S}]Z_\Psi(s_1,\cdots,s_r)$$
	where $s_i=-\pair{\lambda,\alpha_i^\vee}$, $\alpha_1^\vee,\cdots,\alpha_r^\vee$ are the simple coroots.
\end{nthm}
See Theorem \ref{Eis-conj} for a precise statement. This confirmed the so called \emph{Eisenstein conjecture} by Brubaker-Bump-Friedberg.
\tpoint{Eisenstein conjecture}In \cite{WMDS1,WMDS2}, Brubaker-Bump-Friedberg conjectured that the Whittaker coefficients of a metaplectic Eisenstein series are Weyl group multiple Dirichlet series. The main result Theorem \ref{Eis-conj} of this article confirms this conjecture under the mild assumption that the metaplectic dual root datum is of adjoint type. 
\par The Eisenstein conjecture was proved for the root system $A_2$ by Brubaker-Bump-Friedberg-Hoffstein in \cite{WMDS3} by direct computations, then proved by Brubaker-Bump-Friedberg in \cite{BBF:annals} for type $A$ root systems, and subsequently by Friedberg-Zhang in \cite{FZ:ajm} for type $B$ root systems. In these works the key ingredient of the proof is a combinatorial description of the $p$-part in terms of crystal graphs or Gelfand-Tsetlin patterns. Our approach to the Eisenstein conjecture works uniformly for all types of root systems and used a different description of the $p$-part as metaplectic Casselman-Shalika formula due to McNamara \cite{McN:tams} and Chinta-Gunnells \cite{cg:jams}. 
\tpoint{Eisenstein series on metaplectic groups}Let $G$ be a split semisimple simply-connected group over a global field $k$ containing all $n$-th roots of unity. For each place $\nu$, let $G_\nu=G(k_\nu)$, the \emph{metaplectic cover} $\widetilde G_\nu$ of $G_\nu$ is a certain central extension of $G_\nu$ by $\mu_n(k)$, the group of $n$-th roots of unity in $k$. See \S\ref{def:metaplectic cover} for this construction. There is a \emph{global metaplectic group} $\widetilde G_\BA$ which is a central extension of $G_\BA$ by $\mu_n(k)$. See \S\ref{def:Global-metaplectic-group}. 
\par In particular, the group $G_k$ of $k$-rational points splits canonically in $\widetilde G_\BA$, so it makes sense to talk about automorphic forms on $\widetilde G_\BA$, which are functions on $\widetilde G_\BA/G_k$ satisfying some analytic properties. For example, classical modular forms of half-integral weights can be viewed as automorphic forms on the two-fold metaplectic cover of $\SL_2(\BA)$ for the field of rational numbers $\BQ$. In the context of modular forms of half integral weights, the Whittaker coefficients of metaplectic Eisenstein series were first studied by Hecke, which results in the Dirichlet series of a quadratic Dirichlet character. 
\par Kubota first considered coverings groups of $\SL_2$ of degree larger than $2$ in \cite{Ku:jms}. In \cite{Ku:book,Ku:pspum} Kubota computed the Fourier coefficients of an Eisenstein series on the $3$-fold metaplectic cover of $\SL_2(\BA)$ for the number field $k=\BQ(\sqrt{-3})$, and results in a Dirichlet series whose coefficients are cubic Gauss sums. The Gauss sum coefficients are not multiplicative, so this Dirichlet series is not an Euler product as in the $2$-fold cover case. 
\par In the pioneering work \cite{KP:pmihes} of Kazhdan-Patterson they defined the metaplectic $n$-fold covers for $\GL_r$ both locally and globally. They computed the Whittaker coefficients of the Borel Eisenstein series on the $n$-fold cover of $\GL_2(\BA)$ they constructed, and the result is essentially a Gauss sum Dirichlet series as in the computation of Kubota. 
\par There are also some work on the Whittaker coefficients of Eisenstein series of \emph{maximal parabolic} type on metaplectic groups . For example, see \cite{BBL:comp,BF:gfa}. 
\tpoint{Weyl group multiple Dirichlet series}\label{WMDS-intro}
The idea of Weyl group multiple Dirichlet series originates in the study of moments of quadratic $L$-functions. This idea goes back to the work of Goldfeld-Hoffstein \cite{GH:invent} on the first moment, and was summarized in the paper \cite{DGH:comp} of Diaconu-Goldfeld-Hoffstein. Roughly speaking, to estimate the $m$-th moment, they are lead to study a certain multiple Dirichlet series whose group of functional equation is a Coxeter group for the Coxeter diagram with one central node and $m$ nodes connected to the central node. The following image shows the dynkin diagrams for $m=1,2,3$, which are Dynkin diagrams for the root systems of type $A_2$, $A_3$, $D_4$ respectively. 
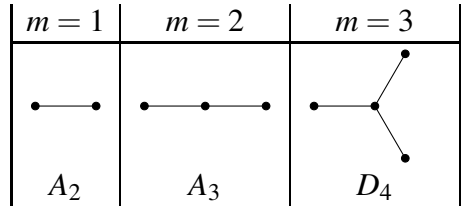
\begin{center}
	\begin{tabular}{ | c | c | c | c |}
	$m=1$ & $m=2$ & $m=3$  \\
	\hline
	\dynkin A2 & \dynkin A3 & \dynkin D4 \\
	$A_2$ & $A_3$ & $D_4$
\end{tabular}
\end{center}
\par Weyl group multiple Dirichlet series (WMDS) are multiple Dirichlet series $Z_\Psi(s_1,\cdots,s_r)$ attached to a root system with nice analytic properties: they have meromorphic continuations to the whole complex space and has group of functional equations isomorphic to the Weyl group of the root system. They were first defined in the series of papers \cite{WMDS1,WMDS2,WMDS3}. Concretely, 
$$Z_\Psi(s_1,\cdots,s_r)=\sum_{C_1,\cdots,C_r\in (\fo_S\setminus\{0\})/\fo_S^\times}H(C_1,\cdots,C_r)\Psi(C_1,\cdots,C_r)\BN C_1^{-s_1}\cdots\BN C_r^{-s_r}$$ 
where the $H$-coefficeints are twisted multiplicative, namely 
\begin{align}\label{twisted-multiplicativity-intro}
	&\nonumber\frac{H(C_1C_1',\cdots,C_rC_r')}{H(C_1,\cdots,C_r)H(C_1',\cdots,C_r')}
	\\=&\prod_{i=1}^r\left(\frac{C_i}{C_i'}\right)_S^{\mathsf Q(\alpha_i^\vee)}\left(\frac{C_i'}{C_i}\right)_S^{\mathsf Q(\alpha_i^\vee)}\prod_{1\leq i<j\leq r}\left(\frac{C_i}{C_j'}\right)_S^{\mathsf B(\alpha_i^\vee,\alpha_j^\vee)}\left(\frac{C_i'}{C_j}\right)_S^{\mathsf B(\alpha_i^\vee,\alpha_j^\vee)}
\end{align}
where $\sQ:\Lambda^\vee\to\BZ$ is the unique Weyl group invariant quadratic form on the coweight lattice $\Lambda^\vee$ normalized such that its value on short coroots is equal to $1$, and $\sB$ is the associated bilinear form by (\ref{bil}). $\Psi$ is another function in $C_1,\cdots,C_r$ that plays a less important role. 
\par  Because of twisted multiplicativity, the $H$-coefficients are determined by the values of $H(\pi_\nu^{k_1},\cdots,\pi_{\nu}^{k_r})$ for every $\nu\notin S$ and $r$-tuple of non-negative integers $k_1,\cdots,k_r$, where $\pi_\nu$ is a prime element of $\fo_S$ corresponding to the place $\nu$. It is convenient to form the generating series
\begin{equation}\label{p-part-intro}
	N(\nu)=\sum_{k_1,\cdots,k_r\geq0}H(\pi_\nu^{k_1},\cdots,\pi_{\nu}^{k_r})e^{-k_1\alpha_1^\vee-\cdots-k_r\alpha_r^\vee}\in\BC[[\Lambda^\vee]]
\end{equation}
which is called the $\nu$-part of the series. Correctly chosed $\nu$-parts will result in the desired analytic continuation and functional equations of the series $Z_\Psi$. 
\par There are several equivalent ways to construct the $\nu$-part: 
\begin{itemize}
	\item The original construction in \cite{WMDS1,WMDS2} by directly define $H(\pi_\nu^{k_1},\cdots,\pi_{\nu}^{k_r})$ as a certain product of Gauss sums. This only works in the \emph{stable case}, namely when $n$ is sufficiently large with respect to the root system.
	\item The combinatorial construction by Brubaker-Bump-Friedberg in \cite{BBF:annals,BBF:book} by crystal graphs. This construction goes beyond the stable case and leads to a prove of the Eisenstein conjecture. 
	\item The construction by Chinta-Gunnells in \cite{cg:jams} by averaging the constant polynomial over a non-standard Weyl group action on $\BC(\Lambda^\vee)$. It was then proved by McNamara that the $\nu$-parts constructed in this way are equal to the values of certain values of a metaplectic Whittaker function on $\widetilde G_\nu$, and these values are given by a metaplectic version of the Casselman-Shalika formula. See also \cite{PP:adv}. 
	\item The construction by a statistical mechanical model by Brubaker-Bump-Chinta-Friedberg-Gunnells in \cite{met-ice} and \cite{BBF:nams}. This model was then intensively studied by Brubaker, Buciumas, Bump, Friedberg, Gray, Gustafsson, and other people in a series of papers. See \cite{BBBG} and reference therein. 
\end{itemize}
\tpoint{Consequences}The work of Moeglin-Waldspurger \cite{MW} on Eisenstein series are applicable for metaplectic covers, so the metaplectic Eisenstein series have analytic continuations and functional equations similar to non-metaplectic Eisenstein series. As a result, the Whittaker coefficient also have analytic continuation and functional equations, so logically the proof of the Eisenstein conjecture leads to a new proof of the analytic properties of Weyl group multiple Dirichlet series. 
\par Nevertheless, we hope that our conputation shed some light in the definition of Weyl group multiple Dirichlet series for \emph{affine} root systems. In the affine case the Weyl group is an infinite group, and in general it is still unknown how to construct the $p$-parts so that the Weyl group multiple Dirichlet series still have analytic continuations and functional equations. Some work in this direction were done by Diaconu, Whitehead, Pasol, Popa, Sawin in \cite{DP,DPP,Wh:comp,Wh,Saw}. 
\par In an ongoing work, we compute the first Whittaker coefficient of a Borel Eisenstein series on an \emph{affine metaplectic group} defined in \cite{PP:duke} for $k=\BF_q(t)$ the rational function field over a finite field. We hope that the analytic properties of Eisenstein series on affine Kac-Moody groups developed by Garland in \cite{Ga1,Ga2,Ga3,Ga4} can be generalized to affine metaplectic groups and help us to establish the analytic properties of the affin Weyl group multiple Dirichlet series. 
\subsection{Local and global Whittaker functionals for metaplectic groups}In this section we introduce the features of the Whittaker coefficients of metaplectic Eisenstein series. A substantial difference with the non-metaplectic case is that the Whittaker coefficients are not factorizable as Euler products, essentially due to failure of local uniqueness of Whittaker models for metaplectic groups. 
\tpoint{Non-metaplectic case}To compare with the metaplectic case, we recall that in the non-metaplectic case, the Whittaker coefficients of a Borel Eisenstein series is factorizable as an Euler product. More concretely, let $G$ be a split simple simply-connected group over $k$ with Borel subgroup $B=TU$ where $T$ is a split torus and $U$ the unipotent radical. Let $\lambda\in X^*(T)\otimes\BC$ be a complex weight, which induces a character $\chi_\lambda:T_\BA\to\BC^*$. The adelic group $G_\BA$ has Iwasawa decomposition $G_\BA=K_\BA T_\BA U_\BA$, and we can define a function $\Phi_\lambda^\flat:G_\BA\to\BC$ by 
$$\Phi_\lambda^\flat(ktu)=\chi_{\lambda+\rho}(t)$$
The function $\Phi_\lambda^\flat=\prod_\nu\Phi_{\lambda,\nu}^\flat$ is factorizable and for $\nu\nmid\infty$ $\Phi_{\lambda,\nu}^\flat$ is the normalized spherical vector in the local principal series representations $I_\nu^\flat(\lambda)$. The Eisenstein series induced from $\Phi_\lambda^\flat$ is 
$$E^\flat(\lambda,g)=\sum_{\gamma\in G_k/B_k}\Phi_\lambda^\flat(g\gamma)$$
By a standard unfolding process, the Whittaker coefficients of $E(\lambda,g)$ for a generic character $\psi$ are given by 
$$W^\flat(\lambda,a)=\int_{U_\BA^-/U_k^-}E^\flat(\lambda,au^-)\psi(u^-)du^-=(\text{archimedean contributions})\cdot\prod_{\nu\nmid\infty}W^\flat_{\lambda,\nu}(a_\nu)$$
where $W^\flat_{\lambda,\nu}$ is a spherical Whittaker function on $G_\nu$, and $W^\flat_{\lambda,\nu}(a_\nu)$ are given by the Casselman-Shalika formula. 
\tpoint{Metaplectic case: non-factorizable}In the metaplectic case, we still want to consider the Eisenstein series induced from a factorizable section of principal series representations, but several complications arise: 
\begin{itemize}
	\item We have to exclude a finite number of "bad" non-archimedean places where the notion of "spherical vector"does not make sense. This corresponds to the choice of the finite set $S$ of places.
	\item For "good" places $\nu\notin S$, the metaplectic torus $\widetilde T_\nu$ is not commutative, so the principal series representations $I_\nu(\lambda)$ are induced not from $\widetilde T$ but from a maximal abelian subgroup $\widetilde T_0$ of $\widetilde T$. We take the factorizable section $\Phi_\lambda$ such that $\Phi_{\lambda,\nu}$ is the normalized spherical vector in $I_\nu(\lambda)$. 
	\item After a similar unfolding computation, the first Whittaker coefficient is not factorizable:
	\begin{equation}\label{after-unfolding}
		W(\lambda,1)=\int_{U_\BA^-/U_k^-}E(\lambda,u^-)\psi(u^-)du^-=\sum_{\eta\in T_k/T_{0,k}}(\text{contributions from }S)\cdot\prod_{\nu\notin S}W_{\lambda,\nu}(\eta)
	\end{equation}
	where $T_0$ is the subgroup of $T$ corresponding to the metaplectic lattice $\Lambda_0^\vee\subseteq\Lambda^\vee=X_*(T)$ defined by (\ref{metaplectic-lattice}), and $W_{\lambda,\nu}(\eta)$ is given by 
	$$q^{\pair{\rho,\lambda^\vee}}\int_{U_\nu^-}\Phi_{\lambda,\nu}(u^-\eta)\psi_\nu(u_\nu^-)du^-$$
	where $\Phi_{\lambda,\nu}$ is a spherical vector. 
\end{itemize}
\tpoint{Local Whittaker functions and $\nu$-parts of WMDS}It turns out that those $W_{\lambda,\nu}(\eta)$ appeared in (\ref{after-unfolding}) are computable. To see this, let $\pi_\nu$ be a uniformizer of $k_\nu$. Essentially by the relation between metaplectic Whittaker functions and the $\nu$-parts of WMDS in \cite{McN:tams}, for any $\lambda^\vee\in\Lambda^\vee$ we have 
\begin{equation}\label{local-Whit=p-part}
	q^{\pair{\rho,\lambda^\vee}}\int_{U_\nu^-}\Phi_{\lambda,\nu}(u^-\pi^{\lambda^\vee}_\nu)\psi_\nu(u_\nu^-)du^-=\sum_{\substack{k_1,\cdots,k_r\geq0\\k_1\alpha_1^\vee+\cdots+k_r\alpha_r^\vee-\lambda^\vee\in\Lambda_0^\vee}}H(\pi_\nu^{k_1},\cdots,\pi_\nu^{k_r})q_\nu^{-k_1s_1-\cdots-k_rs_r}
\end{equation}
where $s_i=-\pair{\lambda,\alpha_i^\vee}$ and $q_\nu$ is the residue characteristic of $k_\nu$. This is a sub-summation of the $\nu$-part (\ref{p-part-intro}) of the WMDS. 
\par For every $\eta\in T_k$ we can factorize $\eta$ in $T_\nu$ as $\pi_\nu^{\lambda_\nu^\vee}\eta^\nu$ for $\eta^\nu\in T_{\ro_\nu}$, then we have $$W_{\lambda,\nu}(\eta)=D(\eta;\nu)\sum_{\substack{k_1,\cdots,k_r\geq0\\k_1\alpha_1^\vee+\cdots+k_r\alpha_r^\vee-\lambda^\vee\in\Lambda_0^\vee}}H(\pi_\nu^{k_1},\cdots,\pi_\nu^{k_r})q_\nu^{-k_1s_1-\cdots-k_rs_r}$$
where $D(\eta;\nu)$ is a root of unity obtained when lifting the factorization $\eta=\pi_\nu^{\lambda_\nu^\vee}\eta^\nu$ to the central extension $\widetilde T_\nu$ of $T_\nu$. 
\tpoint{Two local-to-global "gluing" processes}\label{compatibility-gluing}
As a result, $\prod_{\nu\notin S}W_{\lambda,\nu}(\eta)
$ is equal to $D(\eta):=\prod_{\nu\notin S}D(\eta;\nu)$ times a product of some sub-summations of the $\nu$-parts. In other words, it "glues" the different $\nu$-parts by a factor $D(\eta)$. 
\par Correspondingly, there is another "gluing" process for the WMDS: suppose we already know all the $\nu$-parts for a WMDS, the coefficients $H(C_1,\cdots,C_r)$ then can be "glued" from the coefficients $H(\pi_\nu^{k_1},\cdots,\pi_\nu^{k_r})$ in the $\nu$-parts by twisted multiplicativity. So $H(C_1,\cdots,C_r)$ is a product of $H(\pi_\nu^{k_1^\nu},\cdots,\pi_\nu^{k_r^\nu})$ for $\nu\notin S$ together with a root of unity $D(C_1,\cdots,C_r)$ coming from twisted multiplicativity, where $k_i^\nu=\ord_\nu C_i$. This is done in Lemma \ref{lem:H-fact}. 
\par The crucial point in the proof is that these two gluing processes are \emph{the same}. Concretely, let $\eta_k(C_1,\cdots,C_r)\in T_k$ be the element defined by (\ref{eta-field}), then we have $D(\eta_k(C_1,\cdots,C_r))=D(C_1,\cdots,C_r)$. See Lemma \ref{i-nu-eta-C}. This compatibility of gluing essentially allows us reduce the Eisenstein conjecture --- the global statement that \begin{itemize}
	\item Whittaker coefficients of metaplectic Eisenstein series equals to WMDS
\end{itemize} to the local statement (\ref{local-Whit=p-part}) relating local Whittaker integrals with the $\nu$-parts of WMDS, which are local counterparts of the Whittaker coefficient of metaplectic Eisenstein series and the WMDS respectively.
\par Finally, the Eisenstein conjecture follows the fact that the "contribution from $S$" in (\ref{after-unfolding}) can be matched with the function $\Psi$ in the definition of WMDS. 
\tpoint{Total Whittaker functionals}
There is one more local-global relation that can be seen from the proof of the Eisenstein conjecture. Namely, in the right hand side (\ref{after-unfolding}) we can think of each term of the summation for $\eta\in T_k/T_{0,k}$ as applying a global Whittaker functional $L_\eta$ to the factorizable section $\Phi_{\lambda}$ for the principal series representations, namely the Whittaker coefficient can be viewed as a sum over all the different Whittaker functionals applied to $\Phi_\lambda$. We call this sum of all the Whittaker functional the \emph{total Whittaker functional}, and the Eisenstein conjecture can be stated as 
\begin{equation}\label{global-total-Whit}
	\text{total Whittaker functional applied to }\Phi_{\lambda}=\text{WMDS}
\end{equation}
\par There is a local counterpart of the "total Whittaker functional". In fact, a basis of the space of Whittaker functionals of $I_\nu(\lambda)$ can be given by 
$$L_{\lambda,\mu^\vee}:I_\nu(\lambda)\to\BC,\,\varphi\mapsto q^{\pair{\lambda+\rho,\mu^\vee}}\int_{U_\nu^-}\varphi(u^-\pi^{\mu^\vee}_\nu)\psi_\nu(u_\nu^-)du^-$$
for $\mu^\vee\in \Lambda^\vee/\Lambda_0^\vee$. We form a \emph{total Whittaker functional} $\CW_\lambda=\sum_{\mu^\vee\in \Lambda^\vee/\Lambda_0^\vee}L_{\lambda,\mu^\vee}$. In \cite{McN:tams} it was proved that 
$$\CW_\lambda(\Phi_{\lambda,\nu})=\cs(0;\nu)_\lambda=N(\nu)_\lambda$$
where $\cs(0;\nu)\in\BC[\Lambda^\vee]$ is a metaplectic version of the Casselman-Shalika formula and the subscript $\lambda$ means the evaluation of the polynomial at $\lambda$ by $e^{\mu^\vee}\mapsto q_\nu^{-\pair{\lambda,\mu^\vee}}$, $N(\nu)$ is the $\nu$-part of the WMDS. See (\ref{cs-general}) and \S\ref{Whittaker-functional}, \S\ref{Whit-function} for more details. This is the local counterpart of (\ref{global-total-Whit}), namely 
\begin{equation}\label{local-total-Whit}
	\text{total Whittaker functional applied to }\Phi_{\lambda,\nu}=\nu-\text{part of WMDS}
\end{equation}
\par The local-global correspondence (\ref{global-total-Whit}) (\ref{local-total-Whit}) generalizes the Euler product factorization of Whittaker coefficients of Eisenstein series on non-metaplectic groups into local Casselman-Shalika formulas. Informally speaking, the Eisenstein conjecture (\ref{global-total-Whit}) essentially follows from the compatibility of the two gluing processes in \S\ref{compatibility-gluing} and the purely local result (\ref{local-total-Whit}). The following diagram summarizes all these local-global relations. 
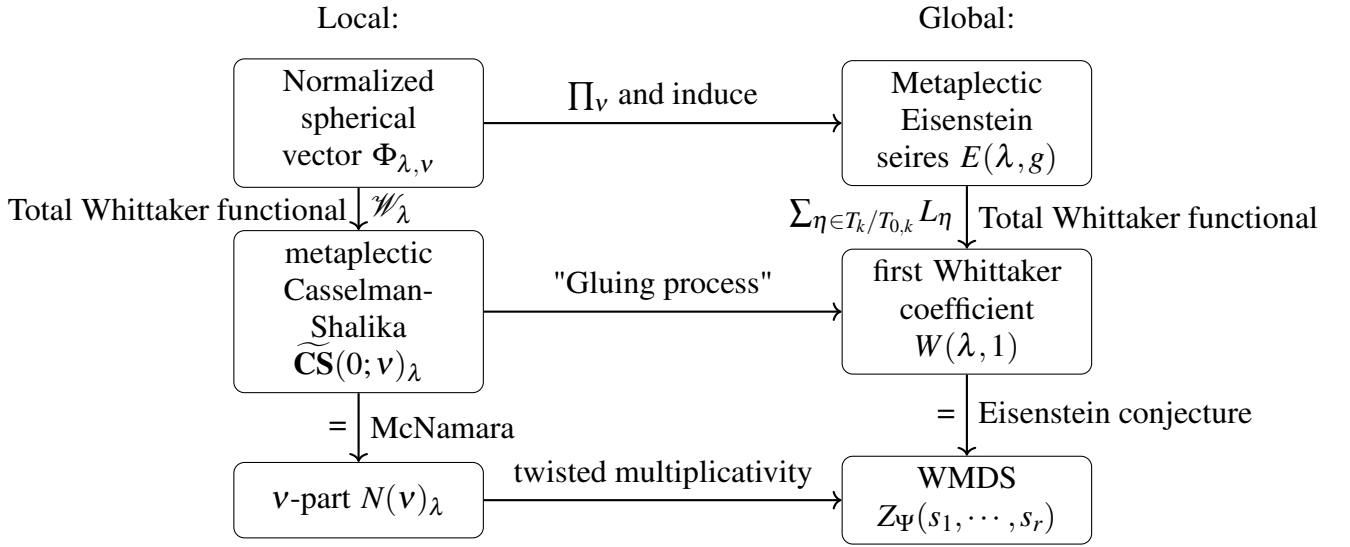
\begin{center}
	\begin{tikzpicture}
		\node (l1) {Local:};
		\node (r1) [right of =l1, xshift=7cm]{Global:};
		\node (l2) [block, below of=l1, yshift=-0.4cm]{Normalized spherical vector $\Phi_{\lambda,\nu}$};
		\node (r2) [block, right of =l2, xshift=7cm]{Metaplectic Eisenstein seires $E(\lambda,g)$};
		\node (l3) [block, below of=l2, yshift=-1.5cm]{metaplectic Casselman-Shalika $\cs(0;\nu)_\lambda$};
		\node (r3) [block, right of =l3, xshift=7cm]{first Whittaker coefficient $W(\lambda,1)$};
		\node (l4) [block, below of=l3, yshift=-1.5cm]{$\nu$-part $N(\nu)_\lambda$};
		\node (r4) [block, right of =l4, xshift=7cm]{WMDS $Z_\Psi(s_1,\cdots,s_r)$};
		\draw [thick,->] (l2)  -- node[anchor=south]{$\prod_\nu$ and induce}(r2);
		\draw [thick,->] (l3)  -- node[anchor=south]{"Gluing process"}(r3);
		\draw [thick,->] (l4)  -- node[anchor=south]{twisted multiplicativity}(r4);
		\draw [thick,->] (l2) -- node[anchor=east]{Total Whittaker functional}node[anchor=west]{$\CW_\lambda$}(l3);
		\draw [thick,->] (r2) -- node[anchor=west]{Total Whittaker functional}node[anchor=east]{$\sum_{\eta\in T_k/T_{0,k}}L_\eta$}(r3);
		\draw [thick,->] (l3) -- node[anchor=east]{=} node[anchor=west]{McNamara} (l4);
		\draw [thick,->] (r3) -- node[anchor=east]{=} node[anchor=west]{Eisenstein conjecture} (r4);
	\end{tikzpicture}
\end{center}
\subsection{Organization of this paper}In Section 2 we introduce the number theoretic concepts that are needed in this paper, including $S$-integers, Hilbert and power residue symbols, and Gauss sums. 
\par In Section 3 we recall the construction of Weyl group multiple Dirichlet series and its $p$-parts. We use the Chinta-Gunnells averaging method to construct the $p$-part, but we follow more closely to the presentation in \cite{PP:adv}. 
\par Section 4 contains the construction of metaplectic covers over non-archimedean local fields, and the basic properties of the metaplectic group and the unramified principal series representations, in particular we investigate the Whittaker functionals carefully. We adopt the viewpoint of universal principal series as in \cite{HKP}. 
\par In Section 5 we recall the construction of global metaplectic groups and define the Eisenstein series we study. Usually the Borel Eisenstein series is induced from a section of the principal series representations, but for bad primes we don't have a good description of the unramified principal series representation for the local metaplectic group, so instead we treat all the bad places in $S$ together. This idea comes from \cite{BBF:annals}. Then we computed the first Whittaker coefficient of this Eisenstein series and thereby proved the Einsenstein conjecture.
\par In Appendix A we proved a relation between the local Whittaker functionals and intertwiners on the universal principal series representation, which is equivalent to the computation of "Kazhdan-Patterson scattering matrix" in the literature. 
\par In Appendix B we formulate a relation between twisted multiplicativity and factorizable genuine functions on the torus. 
\subsection{Acknowledgement}
This work started as a joint one with Manish Patnaik. We thank him for the many discussions we have had about these topics, for his constant help and guidance throughout this work, and for generously sharing his thoughts on related matters. This article would not exist but for them.
\par We would also like to thank Valentin Buciumas, Solomon Friedberg, Dongwen Liu, Sergey Lysenko, Dinakar Muthiah, Anna Pusk\'as and Yongchang Zhu for helpful discussions related to this work and other topics related to metaplectic groups. 

\section{Number theoretic notations}

\newcommand{\places}{\mathscr{V}}

\newcommand{\zee}{\mathbb{Z}}

\subsection{Local and global fields, $S$-integers}
\tpoint{Local fields} \label{notation:local-fields} Let $\lf$ be a non-archimedean local field with discrete valuation $\nu: \lf^* \rr \zee $ and corresponding norm $| \cdot |_{\nu}$.  Define  
\be{} \label{eqn:ring-of-integers-prime-ideal} \begin{array}{lcr}  \ro_\nu:= \{ x \in \lf^* : |x|_\nu \leq 1 \} \cup \{ 0 \} & \text{ and } &  \mathfrak{p}_{\nu} = \{ x \in \lf^* : |x|_\nu < 1 \}\cup \{ 0 \} \end{array} \ee
 as the ring of integers of $\lf$ along with its maximal ideal respectively. Write $ \kappa(\nu) = \ro_{\nu} / \mathfrak{p}_{\nu}$ for the corresponding residue field, a finite field whose cardinality will be denoted as $q_{\nu}.$ Pick a uniformizer $\pi_{\nu} \in \lf$, i.e. $|\pi|_{\nu}=1.$  If there is no danger of confusion, we drop $\nu$ from our notation and write $\ro$ for $\ro_{\nu}$, $\pi$ for $\pi_{\nu}$, \textit{etc.}
\par By an archimedean local field $\lf$ we mean $\BR$ or $\BC$. In this case we define the norm $|\cdot|$ on $\lf$ to be the usual absolute value on $\BR$ or norm on $\BC$.

\tpoint{Global fields} \label{subsub:globalfields} Throughout this work, we let $k$ denote a global field, so either a number field or a function field of a curve $C$ over a finite field. Write $\places_k$, or just $\places$ for the set of (equivalence classes of) valuations of $k.$ In the case that $k$ is a number field, we write $\places_{\infty}$ for the set of archimedean places and $\places_{\fin}$ for the set of finite places. Often we write $\nu \mid \infty$ or $\nu \nmid \infty$ to denote that $\nu \in \places_{\infty}$ or $\nu \notin \places_{\infty}$.

\tpoint{Completions of global fields} \label{subsub:completions-global} For $\nu \in \places_k$, we write $k_\nu$ for the corresponding completion. If $\nu \in \places_{\fin}$, then $k_\nu$ is a non-archimedean local field with ring of integers denoted as $\ro_\nu$. Let $|\cdot|_\nu$ denote the corresponding norm on $k_\nu^*$ and $\pi_\nu \in \ro_\nu$ be a uniformizer. If $\nu \in \places_{\infty}$, then $k_\nu$ is either isomorphic to $\R$ or $\C$, and we say that $k$ is totally real or totally imaginary if all $k_\nu$ are isomorphic to $\R$ or $\C$ respectively.

Fix $k$ a global field. We maintain the same notations as in \S \ref{subsub:globalfields}-\S \ref{subsub:completions-global}.

\tpoint{$S$-integers} Let $S\subset\places_k$ be a finite set of places containing $\places_\infty$. Let 
\renewcommand{\CO}{\o}
\begin{equation}\label{S-integer}
	\CO_S=\{x\in k:|x|_\nu\leq1\text{ for all }\nu\notin S\}
\end{equation}
be the ring of $S$-integers. 
We recall that $\CO_S$ is a Dedekind domain and that for some possibly larger (finite) set of places $S' \supset S$, $\CO_{S'}$ will be a principal ideal domain. We shall also write $\o_S^{\times} \subset \o_S$ for the group of $S$-units, \textit{i.e.}  
\be{} \label{os-units} \o_S^{\times}:= \{  x \in k \setminus \{ 0 \}  \mid |x|_\nu =1 \mbox{ for } \nu \notin S \} \ee

	\tpoint{Conditions on $S$}  \label{subsub:WMDS-S-conditions} Fix a positive integer $n$, and let $k$ be a global field. Assume that $k$ contains all $2n$-th roots of unity, let $S \subset \places_k$ be a finite set of places that satisfies \begin{itemize}
		\item $S \supset \places_{\infty}$;
		\item if $k$ is a number field and the prime ideal of $\nu\in\places_{k}$ is above the (rational) prime factors of $n$, then $\nu\in S$;
		\item if $\nu \in \places_k$ is ramified over $\Q$ (in the number field case) or $\finf(T)$ (in the function field case), then $\nu \in S$;
		\item $\mf{o}_S:= \{ x \in k \mid x_\nu \in \O_\nu \mbox{ for } \nu \notin S \}$ is a principal ideal domain.
	\end{itemize}
	
	\tpoint{Choices of prime elements}  \label{subsub:os-prime-notation} As we have assumed $S$ is sufficiently large so that $\o_S$ is a principal ideal domain, and since the prime ideals in $\o_S$ are in bijection with places $\nu \in \places_k \setminus S$, for every place $\nu\notin S$, the corresponding prime ideal can be given by $\fp_\nu=(\pi_\nu)$ for a single generator $\pi_\nu\in\fo_S$.  We fix once and for all the generators $\pi_\nu$ for all places $\nu\notin S$. Under the map $k \rr k_\nu$, the element $\pi_{\nu} \in \o_S$ is send to a uniformizer in $\O_{\nu}$ of the same name, and in this way we fixed a uniformizer of $\ro_\nu$ for every $\nu\notin S$ at the same time.
	\par Note that if $f(x)$ is any polynomial in one variable, then $f(\pi_{\nu})$ can be regarded as an element of $\o_S$ or $\O_{\nu}$. If however $f(x)$ is a power series in $x$ with infinitely many non-zero coefficents, then $f(\pi_{\nu})$ can only be regarded as an element in $\O_{\nu}$. We hope these comments may clarify any confusion in the notation which may arise. 
	
	\tpoint{A commonly used factorization} \label{subsub:C-notation} 
	We shall often use the following description. Assume $S$ satisfies the conditions of \S  \ref{subsub:WMDS-S-conditions}, and let $C \in \o_S \setminus \{ 0 \}$. If $\nu \notin S$, may write $C =  \pi_{\nu}^m \, C^\nu$  with $C^\nu\in \o_S$ satisfying $\gcd(C^\nu, \pi_{\nu})=1$. The image of $C^\nu \in k_{\nu}$ lies in $\O_{\nu}^{\times}$, so that the image of $C$ in $k_{\nu}$ is equal to 
	\be{} \label{C:nu-factorization} C_{\nu}:= \pi_{\nu}^m \, u \mbox{ with } u \in \O_{\nu}^* \mbox{ the image of } C^\nu \mbox{ in } k_{\nu}.\ee
	So the number $m$ is equal to $\ord_\nu(C)$. We shall sometimes write $m:= n_{\nu}(C)$ and write $C_\nu:=\pi_\nu^{n_\nu(C)}$.

\subsection{Hilbert and Power Residue Symbols} 
In the rest of this section, we fix a positive integer $n$. For any field $\se$, let $\mu_n(\se)$ be the subgroup of $\se^*$ of $n$-th roots of unity in $\se$. We say that $\se$ \emph{contains all $n$-th roots of unity} if the set $\mu_n(\se)$ has cardinality $n$.
\spoint\label{local-Hilbert-symbol}
Let $\lf$ be a local field such which contains all $n$-th roots of unity. Let $(\cdot, \cdot): \lf^* \times \lf^* \rr \mu_n(\lf)$ be the Hilbert symbol as defined in \cite{bump-brubaker:crelle}, namely they are the \emph{inverse} of the Hilbert symbols defined in \cite{ANT}. The main properties of this symbol which we need are summarized as follows (see \cite[Prop. 1, p. 164]{bump-brubaker:crelle}). First of all $(\cdot, \cdot)$ is a \textit{symbol}, i.e. \begin{enumerate}
	\item (Bi-multiplicaticity) $(aa ', b) = (a, b) (a', b)$ and $(a, b b')= (a, b)(a, b')$ for $a, b \in \lf^*$.
	\item (Inverse) $(a, b)^{-1} = (b, a)$
	\item $(a, -a)=1$ and if $a \neq 1$, $(a, 1-a)=1$
\end{enumerate}
Moreover, we have 
\begin{nprop}\label{Hilb-symbol-property}
	\begin{enumerate}[(i)]
		\item If $\lf$ is complex we have $(a, b)=1$ for all $a, b \in \C^*$.
		\item If $\lf$ is non-archimedean with residue characteristic $q$ and $n$ does not divide the residue characteristic of $\lf$ (namely $n|(q-1)$), then $(a, b)=1$ for $a, b \in \O^\times$.
		\item if $\lf$ is non-archimedean and $2n|(q-1)$, then we have $(\pi,\pi)=1$ for the uniformizer $\pi$ of $\lf$. 
	\end{enumerate}
\end{nprop}
 Note that in general $(a, b)=1$ if and only if $a$ is a norm from $\lf(b^{1/n})$.
 
\tpoint{$S$-power residue symbols} Let $k$ be a global field which contains all $n$-th roots of unity. We fix, once and for all, $\epsilon: \mu_n(k) \hookrightarrow \C^*$ an embedding into the complex numbers. For every $\nu\in\places_k$ the embedding $k\hookrightarrow k_\nu$ induces an isomorphism $\mu_n(k)\xrightarrow{\sim}\mu_n(k_\nu)$, and using this isomorphism we identify $\mu_n(k_\nu)$ with $\mu_n(k)$ for every $\nu$, so the Hilbert symbols on $k_\nu$ takes value in $\mu_n(k)$. 
\par For $x,a\in\fo_S$ we define the $S$-power residue symbol following \cite{Deligne:sb} by
\begin{equation}\label{power-residue-symbol}
	 \leg{x}{a}_S = 
\begin{cases}
	\prod_{\substack{\nu\notin S\\\nu\mid a}} (x, a)_\nu &  \text{ if }\gcd(x,a)=1, \\
	0, & \text{otherwise}.
\end{cases}
\end{equation} 
 where $(x, a)_\nu$ is the local Hilbert symbol on $k_\nu$ introduced in \S \ref{local-Hilbert-symbol}.

\tpoint{The symbol $(\cdot, \cdot)_S$} \label{subsub:hilbert-S} Fix the conditions on $k$, $n$, and $S \subset \places_k$ as in \S \ref{subsub:WMDS-S-conditions}. Let $\BA$ be the ring of \adeles of $k$. Define
\be{}  \begin{array}{lccr} k_S=\prod_{v\in S}k_\nu , & \BA_S=\left(\prod_{\nu\notin S}\ro_\nu\right)\times\left(\prod_{\nu\in S}k_\nu\right),& \mbox{and} & \BA^S=\prod_{\nu\notin S}k_\nu. \end{array} \ee
Note that $\o_S=k\cap\BA_S$, $\BA=\BA^S\BA_S$. 

The $n$-th order $S$-Hilbert symbol $(-,-)_S: k_S^\times\times k_S^\times\to\mu_n$ is defined by 
\begin{equation} \label{def:hilb-S} 
	(x,y)_S: =\prod_{\nu\in S}(x,y)_\nu = \prod_{\nu \notin S} \, (x, y)_{\nu}^{-1}.
\end{equation}
The following lemma follows from \cite[Lemma 2]{bump-brubaker:crelle}.
\begin{nlem}[]\label{def:Omega}
\begin{enumerate}[(i)]
	\item $(x,y)_S=1$ for $x,y\in\fo_S^\times$. 
	\item Let $\Omega=\CO_S^\times k_S^{\times,n}$, then $(x,y)_S=1$ for $x,y\in\Omega$. 
\end{enumerate}
\end{nlem}
\subsection{Gauss Sums} 

\tpoint{Additive and multiplicative characters}Let $\lf$ be a non-archimedean local field. Let $\psi:\lf\to\BC^*$ be an additive character. The \emph{conductor} of $\psi$ is the maximal integer $k$ such that $\psi|_{\pi^{-k}\ro}$ is trivial. $\psi$ is called \emph{unramified} if it has conductor $0$. 
\par Note that for any $x\in\lf^*$, the map $y\mapsto (x,y)$ is a multiplicative character of $\lf^*$. 
\tpoint{Gauss sums}
\newcommand{\gs}{\bm{g}}\label{Gauss-sums}
Let $\psi:\lf\to\BC^*$ be an unramified character. Define 
$$G(k,b)=q\int_{\ro^\times}(r,\pi)^{k}\psi(\pi^br^{-1})dr$$
In particular, for $b=-1$, we denote $\bold g_k:=G(k,-1)$. We summarize the properties of Gauss sums as follows: 
\begin{nprop}[\cite{PP:adv}]\label{groperties-Gauss-sum}Suppose $2n|(q-1)$. 
	\begin{enumerate}[(i)]
		\item If $b\leq-2$, then $G(k,b)=0$.
	\item If $b\geq0$, then $$G(k,b)=q\int_{\ro^\times}(r,\pi)^{-k}dr=\begin{cases}
		q-1, & n|k\\
		0, &\text{otherwise.}
	\end{cases}$$
	\item $\bold g_k=\bold g_l$ if $n|(k-l)$;
	\item $\bold g_0=-1$;
	\item $\bold g_k\bold g_{-k}=q$ if $n\nmid k$. 
	\end{enumerate}
\end{nprop}

%

%
%

\section{Weyl group multiple Dirichlet series}
\subsection{Root and Metaplectic datum}\label{notation: root system}
\tpoint{Root data}For a more detailed exposition of this part see \cite{pspum33}. A \emph{root datum} is a quadruple $\fD=(\Lambda,\Phi,\Lambda^\vee,\Phi^\vee)$ where 
\begin{itemize}
	\item $\Lambda$ and $\Lambda^\vee$ are free abelian groups of finite rank, in duality by a pairing $\pair{-,-}:\Lambda\times\Lambda^\vee\to\BZ$. 
	\item $\Phi$ and $\Phi^\vee$ are finite subsets of $\Lambda$ and $\Lambda^\vee$ respectively, there is a bijection $\Phi\to\Phi^\vee$ denoted $\alpha\mapsto\alpha^\vee$. 
\end{itemize}
They are supposed to satisfy the following axioms:
\begin{enumerate}[(RD1)]
	\item $\pair{\alpha,\alpha^\vee}=2$ for every $\alpha\in\Phi$. 
	\item $s_\alpha(\Phi)\subseteq\Phi$ and $s_{\alpha^\vee}(\Phi^\vee)\subseteq\Phi^\vee$, where $s_\alpha:\Lambda\to\Lambda$ is defined by 
	$$s_\alpha(x)=x-\pair{x,\alpha^\vee}\alpha.$$
	and similarly $s_{\alpha^\vee}(\Phi)$ is defined by 
	$$s_{\alpha^\vee}(y)=y-\pair{y,\alpha}\alpha^\vee.$$
\end{enumerate}
Let $W$ be the Weyl group of the root datum. It is identified with the subgroup of $\Aut(\Lambda)$ generated by $\{s_\alpha:\alpha\in\Phi\}$ and the subgroup of $\Aut(\Lambda^\vee)$ generated by $\{s_{\alpha^\vee}:\alpha^\vee\in\Phi\}$. 
\tpoint{Based root data}It follows that $(\Phi,\Lambda\otimes\BQ)$ is a root system. A \emph{based root datum} is a sextuple $\fD=(\Lambda,\Phi,\Delta,\Lambda^\vee,\Phi^\vee,\Delta^\vee)$ such that $(\Lambda,\Phi,\Lambda^\vee,\Phi^\vee)$ is a root datum and $\Delta\subseteq\Phi$ is an ordered basis of the root system $(\Phi,\Lambda\otimes\BQ)$. Since $\Delta\subseteq\Lambda$ and $\Delta^\vee\subseteq\Lambda^\vee$ determines $\Phi$ and $\Phi^\vee$ uniquely, usually a based root datum is denoted as the quadruple $\fD=(\Lambda,\Delta,\Lambda^\vee,\Delta^\vee)$. Usually we denote the ordered set of simple roots $\Delta$ by $\Delta=\{\alpha_1^\vee,\cdots,\alpha_r^\vee\}$ and similarly denote $\Delta^\vee=\{\alpha_1^\vee,\cdots,\alpha_r^\vee\}$.
\par The based root datum is called 
\begin{itemize}
	\item \emph{semisimple} if $\rank\Lambda=\rank\Lambda^\vee=|\Delta|$,
	\item \emph{adjoint} if $\Delta$ is a basis of $\Lambda$,
	\item \emph{simply-connected} if $\Delta^\vee$ is a basis of $\Lambda^\vee$. 
\end{itemize}
\tpoint{Weyl group}For $i=1,\cdots,r$ let $s_i\in W$ be the simple reflection corresponding to $\alpha_i$ or $\alpha_i^\vee$. $(W,\{s_i:i=1,\cdots,r\})$ is a Coxeter system. 
\tpoint{Positive roots and coroots} let $\Phi_+$ (resp. $\Phi_+^\vee$) be the set of positive roots in $\Phi$ with respect to $\Delta$ (resp. positive coroots in $\Phi^\vee$ with respect to $\Delta^\vee$).  Let $\rho=\frac12\sum_{\alpha\in\Phi_+}\alpha\in\Lambda$ be the Weyl vector. It has the property that $\pair{\rho,\alpha_i^\vee}=1$ for every simple coroot $\alpha_i^\vee\in\Delta^\vee$. 

\newcommand{\sA}{{\mathsf A}}

\tpoint{Metaplectic structures}\label{metaplectic-structure}
In the rest of this section we fix a positive integer $n$. Let $\sQ:\Lambda^\vee\to\BZ$ be the unique $\BZ$-valued $W$-invariant quadratic form on $\Lambda^\vee$ such that $\sQ$ takes value $1$ on short coroots. Let $\sB:\Lambda^\vee\times\Lambda^\vee\to\BZ$ be the associated bilinear form defined by
\begin{equation}\label{bil}
	\sB(\lambda^\vee,\mu^\vee)=\sQ(\lambda^\vee+\mu^\vee)-\sQ(\lambda^\vee)-\sQ(\mu^\vee).
\end{equation}
\par We define
\begin{equation}\label{metaplectic-lattice}
	\Lambda_0^\vee=\{\lambda^\vee\in\Lambda^\vee:\sB(\lambda^\vee,\alpha_i^\vee)\equiv0\pmod n\text{ for all }i\in I\}.
\end{equation}
We note that by $W$-invariance of $\sQ$ and $\sB$ we actually have 
$$\sB(\lambda^\vee,\alpha_i^\vee)=\pair{\lambda^\vee,\alpha_i}\sQ(\alpha_i^\vee).$$
For every $\alpha^\vee\in\Phi^\vee$, let $n(\alpha^\vee)$ be the smallest integer such that $n(\alpha^\vee)\mathsf{Q}(\alpha^\vee)$ is a multiple of $n$, and let $\widetilde\alpha=n(\alpha^\vee)\alpha^\vee\in\Lambda^\vee$. 
\begin{nlem}[\cite{McN:ps2012}]\label{metaplectic-dual-root-data}
	\begin{enumerate}[(i)]
		\item $\widetilde\alpha^\vee\in\Lambda_0^\vee$ for every $\alpha^\vee\in\Phi^\vee$ and $\{\widetilde\alpha^\vee:\alpha^\vee\in\Phi^\vee\}$ is a root system in $\Lambda_0^\vee\otimes\BQ$. 
		\item Moreover, let $\Lambda_0\in\Lambda\otimes\BQ$ be the dual lattice of $\Lambda_0^\vee$, namely 
		$$\Lambda_0=\{\lambda\in\Lambda\otimes\BQ:\pair{\lambda,\widetilde\alpha_i^\vee}=1\text{ for }i=1,2,\cdots,r.\}$$
	\end{enumerate}
	Then $(\Lambda_0^\vee, \{n_i\alpha_i^\vee\}_{i=1}^r, \Lambda_0,\{n_i^{-1}\alpha_i\}_{i=1}^r)$ is a based root datum, called the \emph{metaplectic dual root datum} of $(\fD,\sQ,n)$, denoted $\fD_{(\sQ,n)}^\vee$. 
\end{nlem}
\spoint
In the sequel we will use the following notations: we let $\sQ_i:=\sQ(\alpha_i^\vee)$, $\sB_{ij}=\sB(\alpha_i^\vee,\alpha_j^\vee)$, and $n_i:=n(\alpha_i^\vee)$. 

\subsection{The Chinta-Gunnells action}
\tpoint{The generic ring and specializations} Let $\BC_v=\BC[v,v^{-1}]$, $\BC_{v,\Bg}$ be the ring $\BC_v[\Bg_k:k\in\BZ]$ where $\Bg_k$ are formal parameters satisfying the conditions
\begin{itemize}
	\item $\Bg_k=\Bg_l$ if $n|(k-l)$.
	\item $\Bg_0=-1$.
	\item if $n\nmid k$ then $\Bg_{-k}\Bg_k=v^{-1}$. 
\end{itemize}
For every non-archimedean local field $\lf$, let $s_\lf:\BC_{v,\Bg}\to\BC$ be the ring homomorphism defined by 
$$v\mapsto q^{-1},\,\Bg_k\mapsto\bold g_k$$
where $\bold g_k$ are the Gauss sums defined in \S\ref{Gauss-sums}. It induces a map 
$$s_\lf:\BC_{v,\Bg}[\Lambda^\vee]\to\BC[\Lambda^\vee].$$
For $p\in\BC[\Lambda^\vee]$, $s_\lf(p)$ is called the \emph{specialization} of $p$ in $\lf$.
\tpoint{The Chinta-Gunnells action} Let $J$ be the smallest multiplicative-closed subset of $\BC_{v,\Bg}[\Lambda^\vee]$ containing $1-e^{-\widetilde\alpha_i^\vee}$ and $1-ve^{-\widetilde\alpha_i^\vee}$ for every $\alpha^\vee\in \Phi^\vee$. Let $\BC_{v,\Bg}[\Lambda^\vee]_J$ be the localization of $\BC_{v,\Bg}[\Lambda^\vee]$ by $J$. 
\par For each $\lambda^\vee\in\Lambda^\vee$ and $\alpha_i\in\Delta$, following \cite{PP:adv} we define
\begin{equation}\label{Chinta-Gunnells action}
	s_i\star e^{\lambda^\vee}=\frac{e^{s_i\lambda^\vee}}{1-ve^{-\widetilde\alpha_i^\vee}}\left((1-v)e^{\res_{n_i}\left(\frac{\mathsf B(\lambda^\vee,\alpha_i^\vee)}{\mathsf Q(\alpha^\vee_i)}\right)\alpha_i^\vee}-v\Bg_{\mathsf Q(\alpha_i^\vee)+\mathsf B(\lambda^\vee,\alpha_i^\vee)}e^{-\alpha^\vee_i}(e^{\widetilde\alpha_i^\vee}-1)\right)
\end{equation}
where $\res_{n_i}:\BZ\to\{0,1,\cdots,n_i-1\}$ is the residue map for division by $n(\alpha^\vee)$. 
It is proved in \cite{PP:adv} that this formula could be extended to a $W$-action on $\BC_{v,\Bg}[\Lambda^\vee]_J$. We also note that $\frac{\mathsf B(\lambda^\vee,\alpha_i^\vee)}{\mathsf Q(\alpha_i^\vee)}=\pair{\lambda^\vee,\alpha_i}$. 
\par We define the following element in $\BC_{v,\Bg}$:
\begin{equation}\label{cs-general}
	\cs({\lambda^\vee})=q^{-\pair{\rho,\lambda^\vee}}\prod_{\alpha\in R_+}\frac{1-ve^{-n(\alpha^\vee)\alpha^\vee}}{1-e^{-n(\alpha^\vee)\alpha^\vee}}\sum_{w\in W}(-1)^{\ell(w)}\left(\prod_{\beta^\vee\in R^\vee(w^{-1})}e^{-n(\beta^\vee)\beta^\vee}\right)w\star e^{\lambda^\vee}
\end{equation}
This is the metaplectic version of the Casselman-Shalika formula \cite{McN:tams,PP:adv}. 
\subsection{Definition of Weyl group multiple Dirichlet series}\label{WMDS}
Let $k$ be a global field containing all $2n$-th root of unity and $S \subset \places_k$ a finite set of places satisfying the conditions of \S \ref{subsub:WMDS-S-conditions}. Recall also the $S$-power residue symbols $\leg\cdot\cdot_S$ defined in \S \ref{subsub:hilbert-S}. We also fix a based root datum $\fD= (\Lambda, \{ \alpha_i \}_{i\in I}, \Lambda^\vee, \{ \alphav_i\}_{i\in I}  )$ which is semisimple and simply-connected. Let $\sQ$ and $\sB$ be defined as in \S\ref{metaplectic-structure}. Write $r:= |I |$. Fix an embedding $\epsilon:\mu_n(k)\hookrightarrow\BC^*$.

\newcommand{\qs}{\mathsf{Q}}
\par Throughout this article, we will use the following notation for $r$-tuples of objects in a set $X$: by $\ul C\in X^r$ we mean an $r$-tuple of objects $\ul C=(C_1,\cdots,C_r)$ with $C_1\in X$ for $i=1,\cdots,r$.
\tpoint{Twisted multiplicativity} A function $H: \left( \o_S \setminus \{ 0 \} \right)^r\to\BC$ is said to be \textit{twisted multiplicative} if for $C_1,\cdots,C_r,C_1',\cdots,C_r'\in\CO_S\setminus\{0\}$ with $\gcd(C_1\cdots C_r,C_1'\cdots C_r')=1$, we have 
\begin{align}\label{twisted-multiplicativity}
	&\nonumber\frac{H(C_1C_1',\cdots,C_rC_r')}{H(C_1,\cdots,C_r)H(C_1',\cdots,C_r')}
	\\=&\prod_{i=1}^r\epsilon\left(\frac{C_i}{C_i'}\right)_S^{\mathsf Q(\alpha_i^\vee)}\epsilon\left(\frac{C_i'}{C_i}\right)_S^{\mathsf Q(\alpha_i^\vee)}\prod_{1\leq i<j\leq r}\epsilon\left(\frac{C_i}{C_j'}\right)_S^{\mathsf B(\alpha_i^\vee,\alpha_j^\vee)}\epsilon\left(\frac{C_i'}{C_j}\right)_S^{\mathsf B(\alpha_i^\vee,\alpha_j^\vee)}
\end{align}
Note that there are $\epsilon$'s because we defined the power residue symbol to take value in $\mu_n(k)$. 
\par  Recall the conventions from \S \ref{subsub:C-notation}. In particular, for each $C \in \o_S \setminus \{ 0\}$ the integers $n_{\nu}(C)$ are defined for all $\nu \notin S$, and almost all of them are equal to $0$.
\par Since $\CO_S$ is a unique factorization domain by our assumption on $S$, every element $C\in\CO_S\setminus\{0\}$ factorizes uniquely as $$C=a\cdot\prod_{\nu\notin S}\pi_v^{n_\nu(C)}$$
 where $a\in\CO_S^\times$ is an $S$-unit. Such a factorization depends on the choice of $\pi_\nu$ for each $\nu\notin S$, for which we fixed as in \S \ref{subsub:os-prime-notation}. Let $\CO_S^+$ be the subset of $\CO_S$ consisting of elements with $a=1$ in this unique factorization, namely 
 \begin{equation}
 	\fo_S^+=\{\prod_{\nu\notin S}\pi_\nu^{k_\nu}:k_\nu=0\text{ for almost all }\nu\notin S\}\subseteq\fo_S\setminus\{0\}.
 \end{equation}
 Clearly $\CO_S^+$ is a set of representatives of $(\CO_S\setminus\{0\})/\CO_S^\times$. Also note that in the notation of \S\ref{subsub:C-notation}, for every $C\in\CO_S^+$ we have 
 \begin{equation}\label{oS-plus-factorization}
 	C=\prod_{\nu\notin S,\pi_\nu|C}C_\nu
 \end{equation}
 \begin{nlem} \label{lem:H-fact}
 Let $C_1, \ldots, C_r \in \o_S^+$. Then we have 
 \be{} \label{H:fact} H(C_1, \ldots, C_r) = \epsilon( D(C_1, \ldots, C_r)) \, \prod_{\nu \notin S} \, H\left(\pi_{\nu}^{n_\nu(C_1)}, \ldots, \pi_{\nu}^{n_\nu(C_r)}\right) \ee
 where we define 
 \be{} \label{formula_D(C)}D(C_1, \ldots, C_r)=\prod_{\substack{\omega,\nu\notin S\\\omega\neq\nu}} \left(\prod_{i=1}^r (C_{i,\nu},C_{i,\omega}) ^{-\mathsf Q_{i}}_{\nu}\right)\left(\prod_{1\leq i<j\leq r}(C_{i,\omega},C_{j,\nu}) ^{\mathsf B_{ij}}_\nu\right)
\ee
\end{nlem} 
\begin{proof} For every $\ul C=(C_1,\cdots,C_r)\in (\CO_S^+)^r$ we define 
\begin{equation}
	E(\ul C)=\frac{H(C_1,\cdots,C_r)}{\prod_{\nu\notin S}H(\pi_\nu^{n_\nu(C_1)},\cdots,\pi_\nu^{n_\nu(C_r)})}
\end{equation}
So our goal is to prove that $\epsilon(D(\ul C))=E(\ul C)$ for every $\ul C=(C_1,\cdots,C_r)\in (\CO_S^+)^r$. We do this by induction on the size of the set $\Sigma(\ul C)$ of all the prime factors of $C_1,\cdots,C_r$. If $|\Sigma(\ul C)|=1$, clearly by definition we have $\epsilon(D(\ul C))=E(\ul C)=1$. 
\par To finish the induction, we let $\ul C'=(C_1\pi^{l_1},\cdots,C_r\pi^{l_r})$ for $\pi=\pi_{\nu_0}$ with $\nu_0\notin\Sigma(\ul C)$, so that $\Sigma(\ul C')=\Sigma(\ul C)\sqcup\{\nu_0\}$. Then by (\ref{twisted-multiplicativity}) we have 
\begin{align*}
	&H(C_1\pi^{l_1},\cdots,C_r\pi^{l_r})
	\\=&H(C_1,\cdots,C_r)H(\pi^{l_1},\cdots,\pi^{l_r})\prod_{i=1}^r\epsilon\left(\frac{C_i}{\pi^{l_i}}\right)_S^{\mathsf Q_i}\epsilon\left(\frac{\pi^{l_i}}{C_i}\right)_S^{\mathsf Q_i}\prod_{1\leq i<j\leq r}\epsilon\left(\frac{C_i}{\pi^{l_j}}\right)_S^{\mathsf B_{ij}}\epsilon\left(\frac{\pi^{l_i}}{C_j}\right)_S^{\mathsf B_{ij}}
\end{align*}
so by (\ref{power-residue-symbol}) we have
\begin{align*}
	\frac{E(\ul C')}{E(\ul C)}=&\left(\prod_{i=1}^r\epsilon(C_i,\pi^{l_i})_{\nu_0}^{\mathsf Q_i}\right)\left(\prod_{i=1}^r\prod_{\nu|C_i}\epsilon(\pi^{l_i},C_i)_{\nu}^{\mathsf Q_i}\right)
	\\\cdot&\left(\prod_{1\leq i<j\leq r}\epsilon(C_i,\pi^{l_j})_{\nu_0}^{\mathsf B_{ij}}\right)\left(\prod_{1\leq i<j\leq r}\prod_{\nu|C_i}\epsilon(\pi^{l_i},C_j)_{\nu}^{\mathsf B_{ij}}\right)
	\\=&\prod_{\nu\in\Sigma(\ul C)}\left(\prod_{i=1}^r\epsilon(C_{i,\nu},\pi^{l_i})_{\nu_0}^{\mathsf Q_i}\right)\left(\prod_{i=1}^r\epsilon(\pi^{l_i},C_{i,\nu})_{\nu}^{\mathsf Q_i}\right)
	\\\cdot&\left(\prod_{1\leq i<j\leq n}\epsilon(C_{i,\nu},\pi^{l_j})_{\nu_0}^{\mathsf B_{ij}}\right)\left(\prod_{1\leq i<j\leq r}\epsilon(\pi^{l_i},C_{j,\nu})_{\nu}^{\mathsf B_{ij}}\right)
\end{align*}
where we used (\ref{oS-plus-factorization}) and the fact that for $\nu\notin S$ we have $(\ro_\nu^\times,\ro_\nu^\times)_\nu=1$. 
\par On the other hand, by (\ref{formula_D(C)}) we have 
\begin{align*}
	\frac{D(\ul C')}{D(\ul C)}=&\left(\prod_{\substack{\omega=\nu_0\\\nu\in\Sigma(\ul C)}}\left(\prod_{i=1}^r (C_{i,\nu},\pi^{l_i}) ^{-\mathsf Q_{i}}_{\nu}\right)\left(\prod_{1\leq i<j\leq r}(\pi^{l_i},C_{j,\nu}) ^{\mathsf B_{ij}}_\nu\right)\right)
	\\\cdot&\left(\prod_{\substack{\nu=\nu_0\\\omega\in\Sigma(\ul C)}} \left(\prod_{i} (\pi^{l_i},C_{i,\omega})^{-\mathsf Q_{i}}_{\nu_0}\right)\left(\prod_{1\leq i<j\leq r}(C_{i,\omega},\pi^{l_j}) ^{\mathsf B_{ij}}_{\nu_0}\right)\right)
	\\=&\prod_{\nu\in\Sigma(\ul C)}\left(\prod_{i=1}^r (C_{i,\nu},\pi^{l_i}) ^{-\mathsf Q_{i}}_{\nu}\right)\left(\prod_{1\leq i<j\leq r}(\pi^{l_i},C_{j,\nu}) ^{\mathsf B_{ij}}_\nu\right)
	\\\cdot&\left(\prod_{i=1}^r(\pi^{l_i},C_{i,\nu}) ^{-\mathsf Q_{i}}_{\nu_0}\right)\left(\prod_{1\leq i<j\leq r}(C_{i,\nu},\pi^{l_j}) ^{\mathsf B_{ij}}_{\nu_0}\right) 
\end{align*}
So we conclude that whenever $|\Sigma(\ul C')|=|\Sigma(\ul C)|+1$, we have
$$\frac{\epsilon(D(\ul C'))}{\epsilon(D(\ul C))}=\frac{E(\ul C')}{E(\ul C)}$$
and thus by induction, we proved that $\epsilon(D(\ul C))=E(\ul C)$ for every $\ul C=(C_1,\cdots,C_r)\in (\CO_S^+)^r$.

\end{proof}

\tpoint{Defining the Weyl Group Multiple Dirichlet Series (WMDS)} \label{subsub:definition-WMDS}Followng \cite{WMDS1,WMDS2},  Weyl group multiple Dirichlet series associated to $(\mf{D}, n, S, \sQ)$ is a function in complex variables $s_1,\cdots,s_r$ of the form 
\begin{equation}
\label{WMDS-defn}
	Z_\Psi(s_1,\cdots,s_r)=\sum_{0 \neq C_1,\cdots,C_r\in \CO_S/\CO_S^\times}H(C_1,\cdots,C_r)\Psi(C_1,\cdots,C_r)\BN C_1^{-s_1}\cdots\BN C_r^{-s_r}. 
\end{equation}
where $H:(\CO_S-\{0\})^r\to\BC$ is some twisted multiplicative function and $\Psi:(k_S^\times)^r\to\BC$ satisfies
\begin{equation}\label{property-psi}
	\frac{\Psi(\gamma_1 C_1,\cdots,\gamma_r C_r)}{\Psi(C_1,\cdots,C_r)}=\prod_{i=1}^r\epsilon(\gamma_i,C_i)_S^{\mathsf Q_i}\prod_{1\leq i<j\leq r}\epsilon(\gamma_i,C_j)_S^{\mathsf B_{ij}}
\end{equation}
for $C_1,\cdots,C_r\in k_S^\times$ and $\gamma_1,\cdots,\gamma_r\in\Omega$ (we recall that $\Omega$ was defined in Lemma \ref{def:Omega}). In practice, we have a specific choice of $H$ in mind, and allow ourselves some freedom in specifying $\Psi$. 
\par Note that in the definition (\ref{WMDS-defn}), the summation is over the set $((\fo_S\setminus\{0\})/\fo_S^\times)^r$ because the function $H(C_1,\cdots,C_r)\Psi(C_1,\cdots,C_r)\BN C_1^{-s_1}\cdots\BN C_r^{-s_r}$ is invariant under $(\fo_S^\times)^r$. In practice we find it more convenient to sum over the set $(\fo_S^+)^r$ of representatives of $((\fo_S\setminus\{0\})/\fo_S^\times)^r$, namely we actually use
\begin{equation}\label{WMDS-defn2}
	Z_\Psi(s_1,\cdots,s_r)=\sum_{C_1,\cdots,C_r\in \CO_S^+}H(C_1,\cdots,C_r)\Psi(C_1,\cdots,C_r)\BN C_1^{-s_1}\cdots\BN C_r^{-s_r}.
\end{equation}
When $\Psi$ is unchanged throughout the discussion, we drop it from the notation. 
\tpoint{Specifying $H$}\label{def:H-coefficients}
By Lemma \ref{lem:H-fact}, the values of the coefficients $H(C_1,\cdots,C_r)$ for $C_1,\cdots,C_r\in\fo_S^+$ are determined by the values $H(\pi_\nu^{k_1},\cdots,\pi_\nu^{k_r})$ for all places $\nu\notin S$ and $k_1,\cdots,k_r\geq0$. As mentioned above, we have a specific choice in mind for these, and they are defined as follows: for non-negative integers $k_1, \ldots, k_r$, first define the elements $\mc{H}(k_1, \ldots, k_r)  \in \BC_{v,\Bg}$ using
\begin{align*}
	&\cs(0):=q^{-\pair{\rho,\lambda^\vee}}\prod_{\alpha\in R_+}\frac{1-ve^{-n(\alpha^\vee)\alpha^\vee}}{1-e^{-n(\alpha^\vee)\alpha^\vee}}\sum_{w\in W}(-1)^{\ell(w)}\left(\prod_{\beta^\vee\in R^\vee(w^{-1})}e^{-n(\beta^\vee)\beta^\vee}\right)w\star e^{0} 
	\\&= \sum_{k_1,\cdots,k_r\geq 0} \mc{H}(k_1, \ldots, k_r)v^{k_1+\cdots+k_r} e^{-k_1\alpha_1^\vee-\cdots-k_r\alpha_r^\vee}
\end{align*}
Then, for $\nu\notin S$, and $k_1, \ldots, k_r \geq 0$ integers, define $H(\pi_\nu^{k_1},\cdots,\pi_\nu^{k_r})$ as the specialization of $\mc{H}(k_1, \ldots, k_r)$ by setting $v \mapsto q_\nu^{-1}$ and $\mathbb{g}_i \mapsto \mathbf{g}_i$ for $i \in \zee$. For those $k_1, \ldots, k_r \in \zee$ such that $\mc{H}(k_1, \ldots, k_r) $ is not defined, we then take $H(\pi_\nu^{k_1},\cdots,\pi_\nu^{k_r})=0$. For future reference, we define the support of $\cs(0)$ to be the following collection of $\Lv$,  
\be{} \label{supp:cs} \supp\cs(0) = \{ k_1 \alphav_1 + \cdots + k_r \alphav_r  \mid  \mc{H}(k_1, \ldots, k_r) \neq 0 \}. \ee 
Note that this is actually the "inverse" of the support of $\cs(0)$. 
\begin{nexam}
\begin{enumerate}[(i)]
	\item Suppose $n\geq3$. For the root system of type $A_1$, let $\alpha^\vee$ be the unique positive root, then we have 
	$$\cs(0)=1+v\Bg_1 e^{-\alpha^\vee}.$$
	\item Suppose $n\geq4$. For the root system of type $A_2$, let $\alpha_1^\vee,\alpha_2^\vee$ be the two positive simple roots, then we have 
	$$\cs(0)=1+v\Bg_1e^{-\alpha_1^\vee}+v\Bg_1e^{-\alpha_2^\vee}+v^3\Bg_1\Bg_2e^{-2\alpha_1^\vee-\alpha_2^\vee}+v^3\Bg_1\Bg_2e^{-\alpha_1^\vee-2\alpha_2^\vee}++v^4\Bg_1^2\Bg_2e^{-2\alpha_1^\vee-2\alpha_2^\vee}.$$
\end{enumerate}
\end{nexam}
\begin{nrem}
	There are different ways to define the coefficients $H(\pi_\nu^{k_1},\cdots,\pi_\nu^{k_r})$ for $\nu\notin S$ as mentioned in \S\ref{WMDS-intro}. The description given here is essentially using the averaging method of Chinta-Gunnells \cite{cg:jams}. 
\end{nrem}
	
%

\tpoint{The maps $\log^S$ and $\log_\nu$} \label{subsub:log-nu}  
Recall that for each $v \notin S$, let $n_\nu(C_i)$ be defined as in \S \ref{subsub:C-notation}. Define two maps
\be{} && \begin{array}{lcr} \log_\nu: (\o_S^+)^r \rr \Lv,  & \ul C \mapsto \log_\nu \ul C:=  \, \sum_{i=1}^r \, n_\nu(C_i) \alphav_i  \end{array} \\ &&\begin{array}{lcr}  \log^S: (\o_S^+)^r  \rr \oplus_{\nu \notin S} \, \Lv ,&    \ul C \mapsto \log^S \, \ul C:= ( \log_\nu \, \ul C)_{\nu \notin S}  . \end{array} \ee

\tpoint{Regrouping the sum $Z$} \label{subsbub:Z-regroup}By \eqref{formula_D(C)} we see that  $H(\ul C):=H(C_1,\cdots,C_r)=0$ if for some $v\notin S$ we have $\log_v \ul C \notin \, \supp \cs_v(0).$ Hence, if we define  
\be{} \label{def:supp-Z} \supp(Z)=\{\ul C=(C_1,\cdots,C_r)\in(\CO_S^+)^r:\log_v \ul C \in \, \supp \, \cs_v(0) \text{ for all } v\notin S\}. \ee Then we have 
\be{} \label{Z-via-supp} Z(s_1,\cdots,s_r)=\sum_{\ul C\in\supp(Z)}H(C_1,\cdots,C_r)\Psi(C_1,\cdots,C_r)\BN C_1^{-s_1}\cdots\BN C_r^{-s_r}. \ee  
We shall actually need a slight refinement of this description in the proof of the main result Theorem \ref{Eis-conj}. To state it, consider  
\be{}\label{p-Z}
 p_Z: \supp(Z) \stackrel{\log^S}{\longrightarrow} \oplus_{\nu \notin S} \, \Lv \longrightarrow \oplus_{\nu \notin S} \, \Lv/\Lv_0, \ee 
where the latter map is the natural projection (recall the sublattice $\Lv_0 \subset \Lv$ was defined in (\ref{metaplectic-lattice})). For any $\ul\lambda^\vee=(\lambda_\nu^\vee)\in\bigoplus_{\nu\notin S}\Lambda^\vee/\Lambda_0^\vee$, we may describe its fiber under $p$ as 
\begin{equation} \label{fibers:p} 
	 \supp(Z;\ul\lambda^\vee):= p_Z^{-1}(\lv)  = \{\ul C \in\supp(Z):\log_\nu \, \ul C \equiv \, -\lambda_\nu^\vee \pmod{\Lambda_0^\vee}\}
\end{equation}
Then we have a decomposition of $\supp(Z)$ into a disjoint union
$$\supp(Z)=\bigsqcup_{\ul\lambda^\vee\in\bigoplus_{v\notin S}\Lambda^\vee/\Lambda_0^\vee}\supp (Z;\ul\lambda^\vee)$$
For $\ul\lambda^\vee\in\bigoplus_{v\notin S}\Lambda^\vee/\Lambda_0^\vee$, let 
\begin{equation}\label{subsum-Z}
	Z_{\ul\lambda^\vee}(s_1,\cdots,s_r)=\sum_{\ul C\in \supp (Z;\ul\lambda^\vee)}H(C_1,\cdots,C_r)\Psi(C_1,\cdots,C_r)\BN C_1^{-s_1}\cdots\BN C_r^{-s_r}.
\end{equation}
Then we have
\begin{equation}\label{group-sum-Z}
	Z(s_1,\cdots,s_r)=\sum_{\ul\lambda^\vee\in\bigoplus_{v\notin S}\Lambda^\vee/\Lambda_0^\vee}Z_{\ul\lambda^\vee}(s_1,\cdots,s_r).
\end{equation}

%
\spoint For every place $\nu\notin S$ and $\ul C\in(\fo_S^+)^r$ let
\begin{align}\label{D-nu}
	&\nonumber D(\ul C;\nu)=\left(\prod_{i}\epsilon(C_{i,\nu},C_{i}^\nu) ^{-\mathsf Q_{i}}_{\nu}\right)\left(\prod_{i<j} \epsilon(C_{i}^\nu,C_{j,\nu}) ^{\mathsf B_{ij}}_\nu\right)
	\\=&\left(\prod_{i} \prod_{\omega\neq\nu}\epsilon(C_{i,\nu},C_{i,\omega}) ^{-\mathsf Q_{i}}_{\nu}\right)\left(\prod_{i<j} \prod_{\omega\neq\nu}\epsilon(C_{i,\omega},C_{j,\nu}) ^{\mathsf B_{ij}}_\nu\right)
\end{align}

 They are elements of $\mu_n(k)$ that only depend on the prime factorization of $\ul C$, and it is not hard to see that for $\ul C,\,\ul C'\in \supp (Z;\ul\lambda^\vee)$, we have $D(\ul C;\nu)=D(\ul C';\nu)$ for every $\nu\notin S$, so $D(\ul C;\nu)$ depends only on $\ul\lambda^\vee=\log^S\ul C\in\bigoplus_{\nu\notin S}\Lambda^\vee/\Lambda_0^\vee$, and we also denote it by $D(\ul\lambda^\vee;\nu)$, 
\par Note that by (\ref{formula_D(C)}) we have
\begin{equation}\label{D-local-global}
	D(\ul C)=\prod_{\nu\notin S}D(\ul C;\nu)
\end{equation}
for $\ul C\in(\CO_S^+)^r$, So for $\ul C,\,\ul C'\in\supp (Z;\ul\lambda^\vee)$ we also have $D(\ul C)=D(\ul C')$ and we similarly denote $D(\ul C)$ by $D(\ul\lambda^\vee)$. These will be used in the proof of the main result Theorem \ref{Eis-conj} of this paper. 
\section{Local metaplectic groups}
Let $\lf$ be a non-archimedean local field of characteristic $0$, $\bG$ be a simply-connected Chevalley group with root datum $\fD$, $n$ be a positive integer such that $\lf$ contains all $n$-th roots of unity. In this section we will construct the metaplectic central extension of $\bG(\lf)$ by $\mu_n(\lf)$, denoted $\widetilde G$, and study the unramified principal series representations of $\widetilde G$ under a stronger assumption that $2n|(q-1)$, where we recall that $q$ is the residue characteristic of $\lf$. 
\tpoint{Notation conventions} In this section we will use the notations in \S \ref{notation:local-fields} and \S \ref{notation: root system}. In addition, we will use boldface letters $\bA,\bB,\cdots$ to denote group schemes over $\BZ$, and use the corresponding letters $A,B,\cdots$ for the corresponding groups of $\lf$-points $\bA(\lf),\bB(\lf),\cdots$. The groups of $\ro$-points will be denoted $A_\ro,B_\ro,\cdots$. 
\subsection{Chevalley groups over $p$-adic fields}
\newcommand{\sfh}{{\mathsf h}}
For this part we follow \cite{St}.
\tpoint{Chevalley group scheme and subgroups}
Let $\bG$ be the Chevalley group scheme associated to a semisimple simply-connected based root datum $\fD=(\Lambda,\Delta,\Lambda^\vee,\Delta^\vee)$.  It is a group scheme over $\BZ$, and for every field $\se$, the group scheme $\bG\times_{\Spec(\BZ)}\Spec(\se)$ obtained from $\bG$ by base change is a split semisimple simply-connected group scheme over $k$. Let $\bT\subseteq\bB\subseteq\bG$ be the maximal split torus and the Borel subgroup that are constructed along with the Chevalley group $\bG$. 
\tpoint{Root subgroups} For each root $\alpha\in\Phi$ there is a root subgroup $\bU_\alpha\subseteq\bG$, and an isomorphism $\mathsf x_\alpha:\bG_m\xrightarrow{\sim}\bU_\alpha$ that are constructed along with $\bG$. Let $\bU$ be the unipotent radical generated by $\bU_\alpha$ for $\alpha\in \Phi_+$, let $\bU^-$ be the unipotent radical generated by $\bU_\alpha$ for $\alpha\in \Phi_-$. We have semidirect products $\bB=\bU\bT$, $\bB^-=\bU^-\bT$. 
\par we will denote $\mathsf x_{\alpha_i}$ (resp. $\mathsf x_{-\alpha_i}$) by $\mathsf x_i$ (resp. $\mathsf x_{-i}$), and denote $U_{\alpha_i}$ (resp. $U_{-\alpha_i}$) by $U_i$ (resp. $U_{-i}$).
\tpoint{Torus}\label{Chevalley-torus}
The torus $\bT$ is split, and $\Lambda^\vee$ is the cocharacter lattice of $\bT$. For every $\lambda^\vee\in\Lambda^\vee$, the corresponding cocharacter of $\bT$  is denoted $(-)^{\lambda^\vee}:\bG_m\to\bT$.  Also the group functor $\bT$ can be identified with the functor that sends every (unital commutative) ring $R$ into the group $\Lambda^\vee\otimes_\BZ R^\times$. 
\par If $\alpha_i^\vee\in\Delta^\vee$ is a simple coroot, the cocharacter $(-)^{\alpha_i^\vee}$ is also denoted $\mathsf h_i:\bG_m\to\bT$. 
\par Let $\mathsf e$ be any field. Since $\fD$ is simply-connected, $\Lambda^\vee$ is spanned by the simple coroots $\Delta^\vee$, thus every element in $\bT(\mathsf e)$ can be uniquely written as $\sfh_1(F_1)\cdots\sfh_r(F_r)$ for some $F_1,\cdots,F_r\in\se^*$. We introduce the following notation: for an $r$-tuple $\ul F=(F_1,\cdots,F_r)\in(\se^*)^r$, let 
\begin{equation}\label{eta-field}
	\eta_\se(\ul F)=\sfh_1(F_1)\cdots\sfh_r(F_r)\in \bT(\se).
\end{equation}
\tpoint{$p$-adic groups}Let $G=\bG(\lf)$ be the group of $\lf$-points of $\bG$. Then $B,B^-,U,U^-,T$ are all subgroups of $G$. Let $K=G_\ro$. It is a maximal compact subgroup of $G$. 
\begin{nlem}
	The map $\Lambda^\vee\to T_\ro\backslash T,\,\lambda^\vee\mapsto T_\ro\pi^{\lambda^\vee}$ is an isomorphism of abelian groups. 
\end{nlem}
Let $\log_\nu:T/T_\ro\to\Lambda^\vee$ be the inverse of this isomorphism. 
\subsection{Metaplectic covers}\label{def:metaplectic cover}
\tpoint{Steinberg's universal central extension}Let $\CE$ be the universal central extension of $G$ in \cite{St}. Recall that $\CE$ is generated by the symbols $\bold x_\alpha(s)$ for $\alpha\in \Phi,\,s\in\lf$ subject to the relations in \emph{loc. cit. }There is a natural homomorphism $p:\CE\to G$ sending $\bold x_\alpha(s)$ to $\mathsf x_\alpha(s)$, whose kernel is denoted $C$. Then we have a short exact sequence 
\begin{equation}\label{univ-cext}
	0\to C\to \CE\xrightarrow{p} G\to1
\end{equation}
which makes $\CE$ a central extension of $G$. We recall the description of $C$ by Steinberg and Matsumoto. For each $\alpha\in \Phi$ and $t\in \lf^*$ we introduce the following elements in $\CE$: 
$$\bold w_\alpha(t)=\bold x_\alpha(t)\bold x_{-\alpha}(-t^{-1})\bold x_\alpha(t),\,\bold h_\alpha(s)=\bold w_\alpha(s)\bold w_\alpha(1)^{-1}. $$
Then for each $\alpha\in\Phi$ and $s,t\in\lf^*$ there is a central element $\bold c_{\alpha^\vee}(s,t)\in C$ such that 
\begin{equation}\label{relation-T1}
	\bold h_\alpha(s)\bold h_\alpha(t)\bold h_{\alpha}(st)^{-1}=\bold c_{\alpha^\vee}(s,t).
\end{equation}
Steinberg proved that 
\begin{itemize}
	\item For every short (or long) coroot $\alpha^\vee$, the elements $\bold c_{\alpha^\vee}(s,t)$ are equal. Let $\bold c(s,t)$ be the element for short coroots. 
	\item There is a $W$-invariant quadratic form $\mathsf Q:\Lambda^\vee\to\BZ$ such that 
	$$\bold c_{\alpha^\vee}(s,t)=\bold c(s,t)^{\mathsf Q(\alpha^\vee)}.$$
\end{itemize}
In fact $\mathsf Q$ is exactly the quadratic form we defined in \S\ref{metaplectic-structure} under the same notation. 
\par Let $\CT$ be the subgroup of $\CE$ generated by $\bold h_\alpha(s)$ for $\alpha\in\Phi,\,s\in\lf^*$. Let $\mathsf B$ be the quadratic form associated to $\mathsf Q$ by (\ref{bil}). Then one can show that the following relation holds in $\CT$: 
\begin{equation}\label{relation-T2}
	[\bold h_\alpha(s),\bold h_\beta(t)]=\bold c(s,t)^{\mathsf B(\alpha^\vee,\beta^\vee)}.
\end{equation}
\begin{nthm}[\cite{Matsu:asens}]
	$C$ is generated by $\bold c(s,t)$ for $s,t\in\lf^*$ subject to the relations
	$$\bold c(s,-s)=\bold(1-s,s)=1,\,\bold c(s,t)=\bold c(t,s)^{-1},\,\bold c(s,t)\bold c(s,u)=\bold c(s,tu).$$
\end{nthm}
\tpoint{Metaplectic extension of $G$}By pushing out the short exact sequence (\ref{univ-cext}) via the map
$$m:C\to\mu_n(\lf),\,\bold c_{\alpha^\vee}(s,t)\mapsto(s,t)^{\sQ(\alpha^\vee)}$$
given by $n$-th Hilbert symbol $(-,-)$, we obtain a central extension $\widetilde G$ of $G$ by $\mu_n(\lf)$ which fits into the following exact commutative diagram: 
\begin{center}
	\begin{tikzcd}
0 \arrow[r] & C \arrow[r] \arrow[d, "m"] & \CE \arrow[r, "p"] \arrow[d] & G \arrow[r] \arrow[d, no head, double] & 1 \\
0 \arrow[r] & \mu_n(\lf) \arrow[r, "i"]  & \widetilde G \arrow[r, "p"]  & G \arrow[r, rightarrow]                    & 1
\end{tikzcd}
\end{center}
$\widetilde G$ is called the \emph{metaplectic cover} of $G$, and the induced map $\widetilde G\to G$ is also denoted by $p$. Let $\widetilde x_\alpha(s)$ be the image of $\bold x_\alpha(s)$ in $\widetilde G$, similar for $\widetilde w_\alpha(s)$ and $\widetilde h_\alpha(s)$. 
\par For every subgroup $H\subseteq G$, let $\widetilde H=p^{-1}(H)$ be the preimage of $H$ in $\widetilde G$. The subgroup $E$ is said to \emph{split} in $\widetilde G$ if there exists an injective group homomorphism $s:H\to\widetilde G$ which is a section of $p$, namely $p\circ s=\id_H$. In this case we have $\widetilde H\cong\mu_n(\lf)\times H$. 
\tpoint{Splitting of unipotent radicals}The unipotent radicals $U$ and $U^-$ splits canonically in $\widetilde G$: 
\begin{nprop}
	\begin{enumerate}[(i)]
		\item For each $\alpha\in \Phi$, the map $s:U_\alpha\to\widetilde G,\,\mathsf x_\alpha(s)\mapsto\widetilde x_\alpha(s)$ is a injective group homomorphism, and is a section of $p$. 
		\item The map $s:U\to\widetilde G$ obtained by putting the maps in (i) together for $\alpha\in \Phi_+$ is a injective group homomorphism, and is a section of $p$. So $s$ is a canonical splitting of $U$ in $\widetilde G$. The same holds for $U^-$. 
	\end{enumerate}
\end{nprop}
We identify the images of the canonical splitting for $U$ with $U$, and similar for $U^-$. 
\tpoint{Splitting of maximal compact subgroup}
\begin{nprop}[\cite{Moore:pmihes, Sav:angew}]\label{split-maximal-compact}
	If $(q,n)=1$, then there is a splitting of $K$ in $\widetilde G$ whose image in $\widetilde G$ is generated by $\{\widetilde x_\alpha(s):\alpha\in \Phi,s\in\ro\}$. 
\end{nprop}
Note that the splitting of $K$ is not unique in general. We fix the above choice of splitting and identify $K$ with the image of the splitting, namely we view $K$ as a subgroup of $\widetilde G$ via this splitting. 
\tpoint{Metaplectic torus}Let $\widetilde T=p^{-1}(T)$ be the \emph{metaplectic torus}. It is a central extension of $T$ by $\mu_n(\lf)$, and the restriction of $p$ to $\widetilde T$ is given by $\widetilde h_\alpha(s)\mapsto \mathsf h_\alpha(s)$. By (\ref{relation-T1}), (\ref{relation-T2}) we have the following relations in $\widetilde T$: for $\alpha,\beta\in\Phi$ and $s,t\in\lf^*$,
$$[ h_\alpha(s), h_\beta(t)]=(s,t)^{\mathsf B(\alpha^\vee,\beta^\vee)}.$$
\begin{equation}\label{relations-in-T}
	\widetilde h_\alpha(s)\widetilde h_\alpha(t)\widetilde h_{\alpha}(st)^{-1}=(s,t)^{\sQ(\alpha^\vee)}.
\end{equation}
\par Since $\bG$ is simply connected, every element in $T$ can be expressed as $\eta(\ul F)=\mathsf h_1(F_1)\cdots\mathsf h_r(F_r)$ for $\ul F=(F_1,\cdots,F_r)\in\lf^*$ (in this section we always work over the local field $\lf$, so we simply denote $\eta_\lf$ by $\eta$). We define a map $\bs:T\to\widetilde T$ by 
\begin{equation}\label{section-local}
	\bs(\mathsf h_1(F_1)\cdots\mathsf h_r(F_r))=\widetilde h_1(F_1)\cdots\widetilde h_r(F_r)\text{ for }F_1,\cdots,F_r\in\lf^*.
\end{equation}
The map $\bs$ is a section of $p:\widetilde T\to T$, but it is not a group homomorphism in general. 
\par We will often use the following simple observation without mentioning: 
\begin{nlem}\label{factor-zeta-section}
	Every element $\widetilde t\in\widetilde T$ can be uniquely factorized as $\widetilde t=\bs(t)\zeta$ for $\zeta\in\mu_n(\lf)$, $t\in T$. 
\end{nlem}
\begin{proof}
	Take $t=p(\widetilde t)$, then we have $p(\bs(t)^{-1}\widetilde t)=t^{-1}t=1$, so $\bs(t)^{-1}\widetilde t\in\ker p=\mu_n(\lf)$. Let $\zeta=\bs(t)^{-1}\widetilde t$, then $\widetilde t=\bs(t)\zeta$. 
\end{proof}
\tpoint{Center of $\widetilde T$}\label{def:T0}
Let $\bT_0$ be the group functor given by sending a ring $R$ to $\Lambda_0^\vee\otimes_\BZ R^\times$. It is a split torus with cocharacter lattice $\Lambda_0^\vee$. The inclusion $\Lambda_0^\vee\hookrightarrow\Lambda^\vee$ induces an isogeny $\bT_0\to\bT$. By abuse of notation, we denote by $T_0$ the image of the induced map $\bT_0(\lf)\to \bT(F)$. It is subgroup of $T$ generated by elements of the form $s^{\lambda^\vee}$ for $s\in\lf^*$ and $\lambda^\vee\in\Lambda_0^\vee$. Similarly, we define $T_{0,\ro}$ to be the subgroup of $T$ generated by $s^{\lambda^\vee}$ for $s\in\ro^*$ and $\lambda^\vee\in\Lambda_0^\vee$. 
\begin{nprop}[\cite{Weissman:pjm}]\label{center-of-T}
	$Z(\widetilde T)=\widetilde T_0$. 
\end{nprop}
\tpoint{When $(q,n)=1$}In this subsubsection we suppose $(q,n)=1$, so that the splitting of $K$ in Proposition \ref{split-maximal-compact} induces a splitting of $T_\ro$. Since we also assumed that $\lf$ contains all $n$-th roots of unity, we actually have $n|(q-1)$, and thus the Hilbert symbol $(-,-)$ is unramified, namely $(a,b)=1$ for $a,b\in\ro^\times$ (see \S \ref{local-Hilbert-symbol}). This implies that $\bs|_{T_\ro}:T_\ro\to\widetilde T$ is a group homomorphism, and coincides with the splitting of $T_\ro$ induced by Proposition \ref{split-maximal-compact}. 
\par By abuse of notation, for every $\lambda^\vee\in\Lambda^\vee$, the element $\bs(\pi^{\lambda^\vee})$ in $\widetilde T$ is also denoted by $\pi^{\lambda^\vee}$. Then each element $\widetilde a\in\widetilde T$ has a unique decomposition
\begin{equation}\label{decomp-T}
		\widetilde a=\pi^{\lambda^\vee}a'\zeta \text{ where }a'\in T_\ro,\, \lambda^\vee\in\Lambda^\vee,\,\zeta\in\mu_n(\lf).
\end{equation}
 Let $ T_*= T_0T_\ro$. 
In the decomposition (\ref{decomp-T}), we have $T_0=\pi^{\Lambda_0^\vee}T_{0,\ro}$ and $ T_*=\pi^{\Lambda_0^\vee}T_{\ro}$. 
\begin{nlem}
		$\widetilde T_*$ is an abelian subgroup of $\widetilde T$. 
\end{nlem}
\begin{proof}
	By Proposition \ref{center-of-T}, $\widetilde T_0$ is central in $\widetilde T$. Also $T_\ro$ is abelian because the Hilbert symbol is unramified. So $\widetilde T_{*}=\widetilde T_0T_\ro$ is abelian. 
\end{proof}
Actually by \cite{McN:ps2012}, if $2n|(q-1)$, then $\widetilde T_*$ is a maximal abelian subgroup of $\widetilde T$. Logically we don't need this result in this article. 
\tpoint{The section $\bs$}We record the following basic computation, which will be used repeatedly in the sequel: 
\begin{nprop}\label{local-cocycle}
	For $\ul F,\ul F'\in (\lf^*)^r$, we have 
\begin{equation}
	\bs(\eta(\ul F\ul F'))=d(\ul F,\ul F')\bs(\eta(\ul F))\bs(\eta(\ul F'))
\end{equation}
where $d(\ul F,\ul F')\in\mu_n(\lf)$ is given by 
\begin{equation}\label{def:d-}
	d(\ul F,\ul F')=\left(\prod_{i=1}^r(F_i,F_i')^{-\sQ_i}\right)\left(\prod_{1\leq i<j\leq r}(F_i',F_j)^{\sB_{ij}}\right)
\end{equation}
\end{nprop}
\begin{proof}
	By (\ref{relations-in-T}) we have
	\begin{align*}
		&\bs(\eta(\ul F\ul F'))=\widetilde h_1(F_1F_1')\cdots\widetilde h_r(F_rF_r')
		\\=&\left(\prod_{i=1}^r(F_i,F_i')^{-\sQ_i}\right)\widetilde h_1(F_1)\widetilde h_1(F_1')\cdots\widetilde h_r(F_r)\widetilde h_r(F_r')
		\\=&\left(\prod_{i=1}^r(F_i,F_i')^{-\sQ_i}\right)\left(\prod_{1\leq i<j\leq r}(F_i',F_j)^{\sB_{ij}}\right)\widetilde h_1(F_1)\cdots\widetilde h_r(F_r)\widetilde h_1(F_1')\cdots\widetilde h_r(F_r')
		\\=&d(\ul F,\ul F')\bs(\eta(\ul F))\bs(\eta(\ul F')).
	\end{align*}
\end{proof}
\begin{nrem}
	$d(\ul F,\ul F')$ is a cocycle of the central extension $\widetilde T$ of $T$ by $\mu_n(\lf)$ after we identify $T\cong(\lf^*)^r$ by the map $\eta$. 
\end{nrem}
\begin{ncor}\label{section-homomorphism-Lambda0}
	Suppose the metaplectic dual root datum $\fD_{(\sQ,n)}^\vee$ defined by Lemma \ref{metaplectic-dual-root-data} (ii) is of adjoint type, namely the lattice $\Lambda_0^\vee$ is spanned by $n_i\alpha_i^\vee$ for $i=1,2,\cdots,r$. For $t\in T$ and $t'\in T_0$ we have $\bs(tt')=\bs(t)\bs(t')=\bs(t')\bs(t)$.
\end{ncor}
\begin{proof}
	It suffices to prove this for $t'=s^{\lambda^\vee}$ with $s\in\lf^*$, $\lambda^\vee\in\Lambda_0^\vee$. Suppose $\lambda^\vee=k_1\alpha_1^\vee+\cdots+k_r\alpha_r^\vee$, since $\Lambda_0^\vee$ is spanned by $n_i\alpha_i^\vee$ for $i=1,2,\cdots,r$, we have $n_i|k_i$ for $i=1,\cdots,r$. Suppose $t=\eta(\ul F)$, $t'=\eta(\ul F')$, then we have $F_i'=s^{k_i}$ for $i=1,\cdots,r$, so
	$$d(\ul F,\ul F')=\left(\prod_{i=1}^r(F_i,s)^{-k_i\sQ_i}\right)\left(\prod_{1\leq i<j\leq r}(F_i',s)^{k_i\sB_{ij}}\right)$$
	Note that $n|k_i\sQ_i$ because $n_i|k_i$. Also $n|k_i\sB_{ij}$ because $k_i\sB_{ij}=\pair{\alpha_j^\vee,\alpha_i}k_i\sQ_i$. Thus $d(\ul F,\ul F')=1$ and we are done. 
\end{proof}
\tpoint{Iwasawa decomposition}We recall the following analogue of Iwasawa decomposition for metaplectic groups:
\begin{nprop}
	Every element $g\in\widetilde G$ can be written as $g=ktu$ for $k\in K$, $t\in\widetilde T$, $u\in U$. The decomposition is not unique in general, but the class of $t$ in $\widetilde T_\ro\backslash\widetilde T$ is uniquely determined by $g$. 
\end{nprop}
In fact, by (\ref{decomp-T}), every element $g\in\widetilde G$ can be written as $g=k\pi^{\lambda^\vee}\zeta u$ for $k\in K$, $\lambda^\vee\in \Lambda^\vee$, $\zeta\in\mu_n(\lf)$, $u\in U$, and $\lambda^\vee$ is uniquely determined by $g$. 

\subsection{Universal unramified principal series representations}\label{universal-principal-series}
In this subsection we assume that $2n|(q-1)$. We develop the theory of unramified principal series in the same way as in \cite{HKP} for reductive groups, namely we adopt the formalism of "universal principal series". 
\tpoint{Genuine functions}\label{def:genuine-local}
We fix an embedding $\epsilon:\mu_n(\lf)\to\BC^*$. For any central extension of groups
$$1\to \mu_n(\lf)\to\widetilde E\to E\to 0$$
we say a function $f:\widetilde E\to\BC$ on $\widetilde E$ is \emph{$\epsilon$-genuine} if it satisfies $f(\zeta x)=\epsilon(\zeta)f(x)$ for any $\zeta\in\mu_n(\lf),\,x\in\widetilde E$. Since $\epsilon$ is fixed, such functions are usually called genuine for short. 
\tpoint{Universal character} We consider a "universal" unramified character of $T_0$ defined by
\begin{equation}
	\chi_{\univ}: T_0/T_{0,\ro}\to\BC[\Lambda_0^\vee]^\times,\,\pi^{\mu^\vee}\mapsto e^{\mu^\vee}.
\end{equation}
 It induces a genuine character $\widetilde\chi_\univ:\widetilde T_*\to\BC[\Lambda_0^\vee]^\times$ by 
\begin{equation}
	\widetilde\chi_{\univ}(\pi^{\mu^\vee}a'\zeta)=e^{\mu^\vee}\epsilon(\zeta)\text{ for }\mu^\vee\in\Lambda_0^\vee,\,a'\in T_\ro, \zeta\in\mu_n(\lf). 
\end{equation}
It is right $T_\ro$-invariant by definition. 
\par Let $i_{\univ}=\Ind_{\widetilde T_*}^{\widetilde T}(\widetilde\chi_{\univ})$. It is the space of genuine functions $f:\widetilde T\to\BC[\Lambda_0^\vee]$ such that 
 $$f(a\pi^{\mu^\vee}a'\xi)=\epsilon(\xi)f(a)e^{\mu^\vee},\,\zeta\in\mu_n(\lf),a\in \widetilde T,\mu^\vee\in\Lambda_0^\vee,a'\in T_\ro.$$
on which $\widetilde T$ acts by left translations. In particular, these functions are right $T_\ro$-invariant. 
\begin{nlem}\label{TO-invariant}
	Let $f$ be a left $T_\ro$-invariant function in $i_{\univ}$. Then $f$ is uniquely determined by its value at $1$, and $f(\pi^{\mu^\vee}a'\xi)=0$ if $\mu^\vee\notin\Lambda_0^\vee$. 
\end{nlem}
\begin{proof}
	By (\ref{decomp-T}), every element in $\widetilde T$ can be decomposed uniquely as $\pi^{\mu^\vee}a'\zeta$ for $\mu^\vee\in\Lambda^\vee,a'\in T_\ro$, so $f$ is completely determined by its values at $\pi^{\mu^\vee}$ for $\mu^\vee\in\Lambda^\vee$. Let $\gamma\in\ro^*$ be an element such that $\epsilon(\gamma,\pi)$ is a primitive $n$-th root of unity. Then for each simple coroot $\beta^\vee\in\Pi^\vee$, we have
	$$f(\pi^{\mu^\vee})=f(\pi^{\mu^\vee}h_\beta(\gamma))=\epsilon(\pi,\gamma)^{\mathsf B(\mu^\vee,\beta^\vee)}f(h_b(\gamma)\pi^{\mu^\vee})=\epsilon(\pi,\gamma)^{\mathsf B(\mu^\vee,\beta^\vee)}f(\pi^{\mu^\vee}).$$
	So for $\mu^\vee\notin\Lambda_0^\vee$, we can take some $\beta^\vee\in\Pi^\vee$ such that $n\nmid\mathsf B(\mu^\vee,\beta^\vee)$. This forces $f(\mu^\vee)=0$ for $\mu^\vee\notin\Lambda_0^\vee$. For $\mu^\vee\in\Lambda_0^\vee$, we have 
	$f(\pi^{\mu^\vee})=e^{\mu^\vee}f(1)$, so $f$ is totally determined by $f(1)\in\BC[\Lambda_0^\vee]$. 
\end{proof}
\tpoint{Universal unramified principal series} We define the {universal unramified principal series representation} of $\widetilde G$ by $M_{\univ}:=\Ind_{\widetilde TU}^{\widetilde G}(i_{\univ})$ (parabolic induction). By transitivity of induction, we can equivalently define $M_{\univ}:=\Ind_{\widetilde T_*U}^{\widetilde G}(\widetilde\chi_{\univ})$. We will use the second model, so $M_\univ$ is the space of functions $\varphi:\widetilde G\to\BC[\Lambda_0^\vee]$ such that 
\begin{equation}
	\varphi(g\pi^{\mu^\vee} a'\zeta u)=\epsilon(\zeta)q^{\pair{\rho,\mu^\vee}}e^{\mu^\vee}\varphi(g),\,\zeta\in\mu_n(\lf),g\in\widetilde G,\mu^\vee\in\Lambda_0^\vee,a'\in T_\ro,u\in U.
\end{equation}
on which $\widetilde G$ acts by left translations. It is also a $\BC[\Lambda_0^\vee]$-module. 
\begin{nprop}
	Let $\varphi$ be a left $K$-invariant function in $M_{\univ}$. Then $\varphi|_{\widetilde T}$ is a left $T_\ro$-invariant function in $i_{\univ}$. In particular, $\varphi$ is determined by its value at $1$. 
\end{nprop}
Let $\Phi$ be the unique $K$-invariant element in $M_{\univ}$ such that $\Phi(1)=1$. Then $\Phi$ is a genuine left $K$-invariant right $T_\ro U$-invariant function on $\widetilde G$. It follows from the proof of Lemma \ref{TO-invariant} that $\Phi$ is supported on $p^{-1}(K\pi^{\lambda^\vee}U)$ if and only if $\lambda^\vee\in\Lambda_0^\vee$, and the value of $\Phi$ on $K\pi^{\lambda^\vee}U\xi$ is $q^{\pair{\lambda^\vee,\rho}}e^{\lambda^\vee}\epsilon(\xi)$ for $\lambda^\vee\in\Lambda_0^\vee$. 
\tpoint{Back to the unramified principal series representations}\label{usual-principal-series}
For $\lambda\in\Lambda\otimes\BC$, we define an unramified character $\chi_\lambda$ of $T_0$ by 
$$\chi_\lambda(\pi^{\mu^\vee}a')=q^{-\pair{\lambda,\mu^\vee}}\text{ for }\mu^\vee\in\Lambda_0^\vee,\,a'\in T_{0,\ro}.$$
$\chi_\lambda$ extends to an unramified genuine character $\widetilde\chi_\lambda$ of $\widetilde T_*$ by extending trivially on $T_\ro$:
 \begin{equation}
 	\widetilde\chi_\lambda(\pi^{\mu^\vee}a'\zeta)=\epsilon(\zeta)q^{-\pair{\lambda,\mu^\vee}}\text{ for }\mu^\vee\in\Lambda_0^\vee,\,a'\in T_\ro,\,\epsilon(\zeta)\in \mu_n(\lf).
 \end{equation}
 	 Note that $\widetilde\chi_\lambda$ is right $T_\ro$-invariant by definition. 
 \par Let $i(\lambda)=\Ind_{\widetilde T_*}^{\widetilde T}(\widetilde\chi_\lambda)$ be the representation of $\widetilde T$ induced from the character $\widetilde\chi_\lambda$ of $\widetilde T_*$. Concretely $i(\lambda)$ is the space of functions $f:\widetilde T\to\BC$ such that 
 $$f(a\pi^{\mu^\vee}a'\xi)=\epsilon(\zeta)f(a)q^{-\pair{\lambda,\mu^\vee}},\,\zeta\in\mu_n(\lf),a\in \widetilde T,\mu^\vee\in\Lambda_0^\vee,a'\in T_\ro.$$
\par The "usual" unramified principal series representations associated to $\lambda\in\Lambda\otimes\BC$ is defined by $I(\lambda):=\Ind_{\widetilde TU}^{\widetilde G}(i(\lambda))\cong\Ind_{\widetilde T_*U}^{\widetilde G}(\widetilde\chi_\lambda)$. In the second model, it is the space of functions $\varphi:\widetilde G\to\BC$ such that 
\begin{equation}\label{property-Phi-lambda}
	\varphi(g\pi^{\mu^\vee} a'\zeta u)=\epsilon(\zeta)q^{\pair{\rho-\lambda,\mu^\vee}}\varphi(g),\,\zeta\in\mu_n(\lf),g\in\widetilde G,\mu^\vee\in\Lambda_0^\vee,a'\in T_\ro,u\in U.
\end{equation}
on which $\widetilde G$ acts by left translations. We have $I(\lambda)=M_\univ\otimes_{\BC[\Lambda_0^\vee]}\BC_\lambda$, where $\BC_\lambda$ is the one-dimensional $\BC[\Lambda_0^\vee]$-module on which $e^{\mu^\vee}$ acts by $q^{-\pair{\mu^\vee,\lambda}}$. Let $\Phi_\lambda$ be the image of $\Phi$ in $I(\lambda)$. It is the unique $K$-invariant function in $I(\lambda)$ such that $\Phi_\lambda(1)=1$. 
\tpoint{Intertwiners}Let $J$ be the smallest multiplicative subset of $\BC[\Lambda_0^\vee]$ containing the elements $(1-q^{-1}e^{-\widetilde\alpha^\vee})$ and $(1-e^{-\widetilde\alpha^\vee})$ for every $\alpha^\vee\in\Phi^\vee$, let $\BC[\Lambda_0^\vee]_J$ be the completion at $J$. For any $w\in W$, define the intertwiner $I_w:M_\univ\otimes\BC[\Lambda_0^\vee]_J\to M_\univ\otimes\BC[\Lambda_0^\vee]_J$ by 
\begin{equation}\label{Intertwiner}
	I_w(\varphi)(g)=w\int_{U_w}\varphi(guw)du
\end{equation}
where $U_w=U\cap wU^-w^{-1}$ and the $w$ in the front means the $W$-action on $\BC[\Lambda_0^\vee]_J$. Then we have \begin{itemize}
	\item $I_w$ is $\widetilde G$-equivariant.
	\item $I_w(e^{\lambda^\vee}\varphi)=e^{w\lambda^\vee}\varphi$ for $\lambda^\vee\in\Lambda_0^\vee$.
	\item $I_{w_1w_2}=I_{w_1}I_{w_2}$ if $\ell(w)=\ell(w_1)+\ell(w_2)$. 
\end{itemize}
We have the following metaplectic version of the Gindikin-Karpelevich formula:
\begin{nprop}[\cite{McN:duke}]\label{GK-formula}
Let $s_i$ be a simple reflection. The effect of $I_{s_i}$ on the spherical vector $\Phi$ is given by 
	$$I_{s_i}\Phi=\frac{1-q^{-1}e^{\widetilde\alpha^\vee_i}}{1-e^{\widetilde\alpha^\vee_i}}\Phi$$
\end{nprop}
Also for every $\lambda\in\Lambda\otimes\BC$, $I_w$ induces an intertwiner
$I_w(\lambda):I(\lambda)\to I(w\lambda)$
defined by 
$$I_w(\varphi)(g)=\int_{U_w}\varphi(guw)du.$$
which is the usual intertwiner on unramified principal series representations. 
\subsection{Whittaker functionals}\label{Whittaker-functional}
We fix an additive character $\psi:\lf\to\BC^*$  of conductor $\ro$. 
\tpoint{Generic unramified character of $U^-$}For every simple root $\alpha_i\in\Delta$, $\psi$ induces a character of $U_{-i}$ via the isomorphism $\mathsf x_{-i}:\lf\xrightarrow{\sim}U_{-i}$, denoted $\psi_{-i}$. Then we define a character of $U^-$ by 
$$\psi:U^-\twoheadrightarrow U^-/[U^-,U^-]\cong\prod_{i\in I}U_{-i}\xrightarrow{\prod_{i\in I}\psi_{-i}}\BC^*$$
which we still denote by $\psi$ by abuse of notations. 
\tpoint{Universal Whittaker functionals}For every $\BC[\Lambda_0^\vee]$-valued smooth representation $(r,V)$ of $\widetilde G$, let $\Wh(V)$ be the space of linear functionals $L:V\to\BC[\Lambda_0^\vee]$ such that 
\begin{equation}\label{Whittaker-condition}
	L(r(u^-)v)=\psi(u^-)L(v)\text{ for all }v\in V,\,u^-\in U^-.
\end{equation}
\begin{nprop}
	$\Wh(M_\univ)$ is a free $\BC[\Lambda_0^\vee]$-module of rank $|\Lambda^\vee/\Lambda_0^\vee|$. For every $\lambda^\vee\in\Lambda^\vee/\Lambda_0^\vee$ define a linear functional $L_{\lambda^\vee}$ on $I_u$ by 
	\begin{equation}\label{Whit-functional}
		L_{\lambda^\vee}(f)=q^{-\pair{\lambda^\vee,\rho}}\int_{U^-}f(u^-\pi^{\lambda^\vee})\psi(u^-)^{-1}du^-
	\end{equation}
	Let $\Gamma$ be a set of representatives of $\Lambda^\vee/\Lambda_0^\vee$. Then $\{L_{\lambda^\vee}:\lambda^\vee\in\Gamma\}$ is a set of $\BC[\Lambda_0^\vee]$-module generators of $\Wh(M_\univ)$. 
\end{nprop}
\tpoint{The total Whittaker functional $\CW$}
Note that 
	$$L_{\lambda^\vee+\mu^\vee}(f)=e^{\mu^\vee}L_{\lambda^\vee}(f)$$
	for $\mu^\vee\in\Lambda_0^\vee$, so $L_{\lambda^\vee}(f)e^{-\lambda^\vee}\in\BC[\Lambda^\vee]$ only depends on the residue class of $\lambda^\vee$ in $\Lambda_0^\vee$. We also define the "total" Whittaker functional $\CW:M_\univ\to\BC[\Lambda^\vee]$ by 
\begin{equation}\label{total-Whit}
	\CW(f)=\sum_{\lambda^\vee\in\Lambda^\vee/\Lambda_0^\vee}L_{\lambda^\vee}(f)e^{-\lambda^\vee}
\end{equation}
\begin{nprop}\label{surjectivity-Whittaker}
	$\CW:M_\univ\to\BC[\Lambda^\vee]$ is surjective.
\end{nprop}
\begin{proof}The proof is similar to the proof of \cite[Proposition 3]{BBF:annals}. For any $\xi^\vee\in \Lambda^\vee$, it suffices to find a preimage for $e^{\xi^\vee}$. Let $\varphi_{\xi^\vee}$ be the function supported on $U^-_\ro\widetilde TU$ whose value at $k\pi^{\mu^\vee}a\zeta u$ is equal to 
	$$\begin{cases}
		\epsilon(\zeta)q^{\pair{\rho,\mu^\vee}}e^{\xi^\vee+\mu^\vee}, & \text{ if }\xi^\vee+\mu^\vee\in\Lambda_0^\vee,\\
		0 & \text{ otherwise.}
	\end{cases}$$
	for $k\in U^-_\ro$, $\mu^\vee\in\Lambda^\vee$, $a\in T_\ro$, $\zeta\in\mu_n$, $u\in U$. By Iwasawa decomposition, the function $\varphi_{\xi^\vee}$ is well-defined and is a member of $M_\univ$. Now we consider $$L_{\lambda^\vee}(\varphi_{\xi^\vee})=q^{-\pair{\lambda^\vee,\rho}}\int_{U^-}\varphi_{\xi^\vee}(u^-\pi^{\lambda^\vee})\psi(u^-)^{-1}du^-.$$
	Since the support of $\varphi_{\xi^\vee}$ is in $U_\ro^-\widetilde TU$, $\varphi_{\xi^\vee}(u^-\pi^{\lambda^\vee})=0$ unless $u^-\in U_\ro^-$, so the integration is actually taken over $U_\ro^-$, on which the character $\psi$ is trivial. Thus
	\begin{align*}
		&L_{\lambda^\vee}(\varphi_{\xi^\vee})=q^{-\pair{\lambda^\vee,\rho}}\int_{U_\ro^-}\varphi_{\xi^\vee}(u^-\pi^{\lambda^\vee})du^-
		\\=&q^{-\pair{\lambda^\vee,\rho}}\varphi_{\xi^\vee}(\pi^{\lambda^\vee})=\begin{cases}
			e^{\lambda^\vee+\xi^\vee} & \text{ if }\xi^\vee+\lambda^\vee\in\Lambda_0^\vee,\\
			0 & \text{otherwise.}
		\end{cases}
	\end{align*}
	As a result, $\CW(\varphi_{\xi^\vee})=L_{-\xi^\vee}(\varphi_{\xi^\vee})e^{\xi^\vee}=e^{\xi^\vee}$.
\end{proof}
\subsection{Unramified Whittaker functions}\label{Whit-function}
\tpoint{The unramified Whittaker functions}
The unramified metaplectic Whittaker functions are the functions 
\begin{equation}\label{Whit-fun-defn}
	W_{\lambda^\vee}(g)=q^{-\pair{\lambda^\vee,\rho}}\int_{U^-}\Phi(gu^-\pi^{\lambda^\vee})\psi(u^-)^{-1}du^-.
\end{equation}
for $\lambda^\vee\in\Lambda^\vee$. They satisfy 
\begin{equation}\label{Whit-fun}
	W_{\lambda^\vee}(kgu^-\zeta)=\psi(u^-)\epsilon(\zeta)W_{\lambda^\vee}(g)\text{ for }k\in K,\, g\in \widetilde G,\,u^-\in U^-,\,\zeta\in\mu_n(\lf). 
\end{equation}
and
\begin{equation}
	W_{\lambda^\vee+\mu^\vee}(g)=e^{\mu^\vee}W_{\lambda^\vee}(g)
\end{equation}
thus $W_{\lambda^\vee}(g)e^{-\lambda^\vee}$ depends only on the residue class of $\lambda^\vee$ in $\Lambda^\vee/\Lambda_0^\vee$. Also note that 
$$W_{\lambda^\vee}(g)=L_{\lambda^\vee}(g\cdot\Phi)$$
where $g\cdot\Phi$ is the $\widetilde G$-action on the universal principal series $M_\univ$ by left translations. 
\subsubsection{The "total" Whittaker function}In \cite{PP:adv} Patnaik-Puskas also defined a "total" Whittaker function $W:\widetilde G\to\BC[\Lambda^\vee]$ by 
\begin{equation}
	W(g)=\int_{U^-}\widetilde\Phi(gu^-)\psi(u^-)du^-
\end{equation}
where $\widetilde\Phi:\widetilde G\to\BC[\Lambda^\vee]$ is the function whose value on $K\pi^{\lambda^\vee}U\xi$ is $q^{\pair{\lambda^\vee,\rho}}e^{\lambda^\vee}\epsilon(\xi)$. The values of $\widetilde\Phi$ and $\Phi$ coincide on $K\pi^{\lambda^\vee}U\xi$ for $\lambda^\vee\in\Lambda_0^\vee$. The main result of loc. cit. is
\begin{nthm}
	For every $\lambda^\vee\in\Lambda_+^\vee$ we have $W(\pi^{\lambda^\vee})=\cs(\lambda^\vee)$. 
\end{nthm}
The relation between the Whittaker functions $W_\lambda^\vee$ for $\lambda^\vee\in\Lambda^\vee$ and Patnaik-Puskas' Whittaker function $W(g)$ is 
\begin{nprop}
	We have 
	$$W(g)=\sum_{\lambda^\vee\in\Lambda^\vee/\Lambda_0^\vee}e^{-\lambda^\vee}W_{\lambda^\vee}(g).$$
	In particular, $$W(1)=\CW(\Phi)=\sum_{\lambda^\vee\in\Lambda^\vee/\Lambda_0^\vee}e^{-\lambda^\vee}L_{\lambda^\vee}(\Phi).$$
\end{nprop}
\begin{proof}
	The functions $W$ and $W_{\lambda^\vee}$ all satisfy (\ref{Whit-fun}), so by Iwasawa decomposition, it suffices to prove this theorem for $g=\pi^{\eta^\vee}$. For simplicity we prove it for $g=1$ (and in this paper we only need this result for $g=1$), the proof for a general $\pi^{\eta^\vee}$ is similar. 
	\par By Iwasawa decomposition, we have 
	\begin{align*}
		W(1)&=\int_{U^-}\widetilde\Phi(u^-)\psi(u^-)du^- =\sum_{\mu^\vee\in\Lambda^\vee}\sum_{\xi\in\mu_n(\CK)}\int_{K\pi^{\mu^\vee}U\xi\cap U^-}\epsilon(\xi)q^{\pair{\mu^\vee,\rho}}e^{\mu^\vee}\psi(u^-)du^-
		\\&=\sum_{\mu^\vee\in\Lambda^\vee}q^{\pair{\mu^\vee,\rho}}e^{\mu^\vee}\sum_{\xi\in\mu_n(\CK)}\epsilon(\xi)\int_{K\pi^{\mu^\vee}U\xi\cap U^-}\psi(u^-)du^-.
	\end{align*}
	And similarly, 
	\begin{align*}
		W_{\lambda^\vee}(1)&=q^{-\pair{\lambda^\vee,\rho}}\int_{U^-}\Phi(u^-\pi^{\lambda^\vee})\psi(u^-)du^-
		\\&=q^{-\pair{\lambda^\vee,\rho}}\sum_{\mu^\vee\in\Lambda^\vee}\sum_{\xi\in\mu_n(\CK)}\int_{K\pi^{\mu^\vee}U\xi\cap U^-}\Phi(u^-\pi^{\lambda^\vee})\psi(u^-)du^-
		\\&=\sum_{\mu^\vee\in-\lambda^\vee+\Lambda_0^\vee}q^{\pair{\mu^\vee,\rho}}e^{\mu^\vee+\lambda^\vee}\sum_{\xi\in\mu_n(\CK)}\epsilon(\xi)\int_{K\pi^{\mu^\vee}U\xi\cap U^-}\psi(u^-)du^-
	\end{align*}
	since $\Phi$ is supported on $\Lambda_0^\vee$. Thus
	\begin{align*}
		&\sum_{\lambda^\vee\in\Lambda^\vee/\Lambda_0^\vee}e^{-\lambda^\vee}W_{\lambda^\vee}(1)
	\\&=\sum_{\lambda^\vee\in\Lambda^\vee/\Lambda_0^\vee}\sum_{\mu^\vee\in-\lambda^\vee+\Lambda_0^\vee}q^{\pair{\mu^\vee,\rho}}e^{\mu^\vee}\sum_{\xi\in\mu_n(\CK)}\epsilon(\xi)\int_{K\pi^{\mu^\vee}U\xi\cap U^-}\psi(u^-)du^-
	\\&=\sum_{\lambda^\vee\in\Lambda^\vee/\Lambda_0^\vee}\sum_{\mu^\vee\in-\lambda^\vee+\Lambda_0^\vee}q^{\pair{\mu^\vee,\rho}}e^{\mu^\vee}\sum_{\xi\in\mu_n(\CK)}\epsilon(\xi)\int_{K\pi^{\mu^\vee}U\xi\cap U^-}\psi(u^-)du^-
	\\&=\sum_{\mu^\vee\in\Lambda^\vee}q^{\pair{\mu^\vee,\rho}}e^{\mu^\vee}\sum_{\xi\in\mu_n(\CK)}\epsilon(\xi)\int_{K\pi^{\mu^\vee}U\xi\cap U^-}\psi(u^-)du^-
	\\&=W(1).
	\end{align*}
\end{proof}
In particular, $e^{-\lambda^\vee}W_{\lambda^\vee}(1)$ is the part of $W(1)$ supported on $-\lambda^\vee+\Lambda_0^\vee$. 
\begin{ncor}\label{Whittaker-H}
	We have
	$$e^{-\lambda^\vee}L_{\lambda^\vee}(\Phi)=e^{-\lambda^\vee} W_{\lambda^\vee}(1)=\sum_{\substack{k_1,\cdots,k_r\geq0\\k_1\alpha_1^\vee+\cdots+k_r\alpha_r^\vee-\lambda^\vee\in\Lambda_0^\vee}}H(\pi^{k_1},\cdots,\pi^{k_r})q^{-k_1-\cdots-k_r}e^{-k_1\alpha_1^\vee-\cdots-k_r\alpha_r^\vee}$$
\end{ncor}
\begin{proof}
	This follows from Theorem 4.5.2, Theorem 4.5.3 and the definition of $H(\pi^{k_1},\cdots,\pi^{k_r})$ in \S\ref{def:H-coefficients}. 
\end{proof}
\subsection{Kazhdan-Patterson's scattering matrix}
We have the following result on the relation between the total Whittaker functional $\CW$ and the intertwiner $I_{s_i}$ for a simple reflection $s_i$. It is equivalent to the computation of Kazdhan-Patterson's scattering matrix \cite{KP:pmihes,GSS}.
\begin{nprop}\label{Intertwiner-Whittaker}
For a simple reflection $s_i$, for any $\Phi_\infty\in M_{\univ,\infty}$ we have
	$$\CW(I_{s_i}\Phi_{\infty})=\frac{1-q^{-1}e^{-\widetilde\alpha_i^\vee}}{1-e^{\widetilde\alpha_i^\vee}}s_i\star\CW(\Phi_\infty)$$
	where $s_i\star-$ is the Chinta-Gunnells action defined by (\ref{Chinta-Gunnells action}). 
\end{nprop}
The proof of this proposition will be given in the appendix. 
\subsection{Invariance of $\cs(0)$ under the Chinta-Gunnells action}
\begin{nthm}[\cite{cg:jams}]
	Let $$D=\prod_{\alpha\in\Phi_+}(1-q^{-1}e^{-\widetilde\alpha^\vee})$$
Then $D^{-1}\cs(0)$ is invariant under the Chinta-Gunnells action (\ref{Chinta-Gunnells action}), namely $s_i\star(D^{-1}\cs(0))=D^{-1}\cs(0)$ for every simple reflection $s_i$. 
\end{nthm}
\begin{proof}
	We have $\cs(0)=W(1)=\CW(\Phi)$. Consider $\CW(I_{s_i}\Phi)$. On the one hand, by the Gindikin-Karpelevich formula (\ref{GK-formula}) we have 
	$$\CW(I_{s_i}\Phi)=\frac{1-q^{-1}e^{\widetilde\alpha^\vee_i}}{1-e^{\widetilde\alpha^\vee_i}}\CW(\Phi)=\frac{1-q^{-1}e^{\widetilde\alpha^\vee_i}}{1-e^{\widetilde\alpha^\vee_i}}\cs(0).$$
	On the other hand, by Proposition \ref{Intertwiner-Whittaker} we also have 
	$$\CW(I_{s_i}\Phi)=\frac{1-q^{-1}e^{-\widetilde\alpha^\vee_i}}{1-e^{\widetilde\alpha^\vee_i}}s_i\star\CW(\Phi)=\frac{1-q^{-1}e^{-\widetilde\alpha^\vee_i}}{1-e^{\widetilde\alpha^\vee_i}}s_i\star\cs(0)$$
	Thus we have 
	$$s_i\star\cs(0)=\frac{1-q^{-1}e^{\widetilde\alpha^\vee_i}}{1-q^{-1}e^{-\widetilde\alpha^\vee_i}}\cs(0)$$
	The result follows from the fact that $s_i\star D=s_i D$ and 
	$$s_i D/D=\frac{1-q^{-1}e^{\widetilde\alpha^\vee_i}}{1-q^{-1}e^{-\widetilde\alpha^\vee_i}}.$$
\end{proof}

\section{Whittaker coefficients of metaplectic Eisenstein series}
Let $k$ be a number field containing all $2n$-th roots of unity (in particular, $k$ is totally imaginary). Let $\BA$ be the ring of \adeles of $k$. Let $\bG$ be a split semisimple simply-connected reductive group over $k$ with root datum $\fD$. In this section we define the global metaplectic group $\widetilde G_\BA$, which is a central extension of the adelic group $G_\BA$ by $\mu_n(k)$. The group $G_k$ of $k$-rational points in $G_\BA$ splits canonically in $\widetilde G_\BA$, so we can talk about automorphic forms on $G_k\backslash\widetilde G_\BA$. In particular, we define certain Eisenstein series on $\widetilde G_\BA$ induced from the Borel subgroup. The main result of this paper is the computation of the first Whittaker coefficient of this Eisenstein series, which results in a Weyl group multiple Dirichlet series as defined in (\ref{WMDS-defn}), as conjectured by Brubaker-Bump-Friedberg \cite{WMDS2}. 
\subsection{Global metaplectic covers}
\spoint For any place $\nu$ of $k$, the inclusion $k\hookrightarrow k_\nu$ induces an isomorphism $\mu_n(k)\cong\mu_n(k_\nu)$. Let $\mu_n(\BA)=\bigoplus_\nu\mu_n(k_\nu)$. There is a canonical map $m:\mu_n(\BA)\to\mu_n(k)$ obtained by identifying all the components with $\mu_n(k)$, and take the product of all the components. 
\tpoint{} For each finite place $\nu\nmid\infty$, let $\widetilde G_\nu$ be the metaplectic central extension of $G_\nu=\bG(k_\nu)$ defined in \S\ref{def:metaplectic cover}. For $\nu\nmid n$, let $K_\nu\subseteq\widetilde G_\nu$ be a splitting of the maximal compact subgroup $\bG(\ro_\nu)$ of $G_\nu$. For $\nu|\infty$, let $\widetilde G_\nu=G_\nu\times\mu_n(k)$ be the trivial central extension of $\widetilde G_\nu$ by $\mu_n(k_\nu)$. (Note that by our assumption, $k$ is totally imaginary, so for each $\nu|\infty$, $k_\nu\cong\BC$ and $G(\BC)$ has no nontrivial central extension by $\mu_n(k)$.) 
\par Let $G_\BA=\prod_\nu'G_\nu$ be the restricted direct product of $\{G_\nu:\nu\in\places_k\}$ with respect to $\{K_\nu=\bG(\ro_\nu):\nu\nmid\infty\}$. Let $\prod_\nu'\widetilde G_\nu$ be the restricted direct product of $\{\widetilde G_\nu:\nu\in\places_k\}$ with respect to $\{K_\nu:\nu\nmid \infty,\,\nu\nmid n\}$. The maps $p_\nu:\widetilde G_\nu\to G_\nu$ induces a map $p_\BA'=\prod_\nu p_\nu:\prod_\nu'\widetilde G_\nu\to G_\BA$, which makes $\prod_\nu'\widetilde G_\nu$ a central extension of $G_\BA$. The kernel of $p'_\BA$ is isomorphic the restricted direct product of $\{\mu_n(k_\nu):\nu\in\places_k\}$ with respect to the trivial subgroups $\{1\}$ over the finite places not dividing $n$, so $\ker p'_\BA\cong\mu_n(\BA)$, and we will always make this identification. 
\tpoint{}\label{def:Global-metaplectic-group}
Let $\widetilde G_\BA$ be the pushout of the central extension $p'_\BA:\prod_v'\widetilde G_v\to G_\BA$ via $m:\mu_n(\BA)\to\mu_n(k)$, which fits into the following exact commutative diagram: 
\begin{center}
	\begin{tikzcd}
1 \arrow[r] & \mu_n(\BA) \arrow[r] \arrow[d, "m"] & \prod_\nu'\widetilde G_\nu \arrow[r, "p'_\BA"] \arrow[d, "m
'"] & G_\BA \arrow[r] \arrow[d, no head, double] & 1 \\
1 \arrow[r] & \mu_n(k) \arrow[r]                  & \widetilde G_\BA \arrow[r, "p_\BA"]                 & G_\BA \arrow[r]                                & 1
	\end{tikzcd}
\end{center}
So the natural map $p_\BA:\widetilde G_\BA\to G_\BA$ makes $\widetilde G_\BA$ a central extension of $G_\BA$ by $\mu_n(k)$. 
\begin{nprop}\cite{Moore:pmihes}
	The group of $k$-rational points $G_k:=\bG(k)$ splits canonically in $\widetilde G_\BA$. We view $G_k$ as a subgroup of $\widetilde G_\BA$ via this splitting. 
\end{nprop}
\tpoint{Local metaplectic groups as subgroups of $\widetilde G_\BA$}The global metaplectic group $\widetilde G_\BA$ is not a restricted direct product of $\widetilde G_\nu$ for $\nu\in\places_k$, but we still have the following:
\begin{nprop}
	The composition of the natural embedding $\widetilde G_\nu\hookrightarrow\prod_\nu'\widetilde G_\nu$ with $m':\prod_\nu'\widetilde G_\nu\to\widetilde G_\BA$ is an injection. We view $\widetilde G_\nu$ as a subgroup of $\widetilde G_\BA$ with respect to this injection. 
\end{nprop}
By abuse of notations, given $g_\nu\in\widetilde G_\nu$ for every $\nu\in\places_k$, the element $m'(\prod_\nu g_\nu)$ in $\widetilde G_\BA$ is also denoted by $\prod_\nu g_\nu$ or $(g_\nu)_\nu$. 
\tpoint{Unipotent radicals}Let $U_\BA$ be the subgroup of $G_\BA$ defined by
$$U_\BA=\{(u_\nu)_\nu\in G_\BA:u_\nu\in U_\nu\text{ for all }\nu\in\places_k\}$$
and similarly define $U_\BA^-$. Then $U_\BA=\prod_\nu' U_\nu$ and $U_\BA^-=\prod_\nu U_\nu'$ are restricted direct products of the local unipotent radicals. Both $U_\BA$ and $U_\BA^-$ split canonically in $\widetilde G_\BA$, and we will view them also as subgroups of $\widetilde G_\BA$ via the canonical splitting. 
\tpoint{Torus} Let $\widetilde T_\BA=p_\BA^{-1}(T_\BA)$. It is also a central extension of $T_\BA$ by $\mu_n(k)$. The canonical splitting of $G_k$ in $\widetilde G_\BA$ induces a splitting $\bs_k: T_k\to\widetilde T_\BA$ given by
\begin{equation}\label{section-global}
	\bs_k(\eta)=\prod_\nu i_\nu(\eta)\text{ for }\eta\in T_k.
\end{equation}

\tpoint{Genuine functions}We fix an injection $\epsilon:\mu_n(k)\to\BC^*$, then we can define $\epsilon$-genuine functions on $\widetilde G_\BA$ or $\widetilde T_\BA$ in the same way as in \S \ref{def:genuine-local}, namely a function $\varphi:\widetilde G_\BA\to\BC$ is called $\epsilon$-genuine if $\varphi(\zeta g)=\epsilon(\zeta)\varphi(g)$ for $\zeta\in\mu_n(k)$, $g\in\widetilde G_\BA$, and similar for $\widetilde T_\BA$. Since $\epsilon$ is fixed, such functions are simply called genuine. 
\tpoint{Factorizable functions}Let $\widetilde G_{\BA^S}=\prod_{\nu\notin S}\widetilde G_\nu$.  Note that $\epsilon$ induced injections $\epsilon:\mu_n(k_\nu)\to\BC$ for every $\nu\in\places_k$, so it makes sense to talk about genuine functions on each $\widetilde G_\nu$. Clearly the restriction of a genuine function on $\widetilde G_{\BA^S}$ to $\widetilde G_\nu$ is genuine. Conversely we have
\begin{nprop}\label{genuine-local-global}
	For every $\nu\notin S$, let $\varphi_\nu:\widetilde G_\nu\to\BC$ be a genuine function on $\widetilde G_\nu$ such that the product $\prod_{\nu\notin S}\varphi_\nu:\prod_{\nu\notin S}'\widetilde G_\nu\to\BC$ is a well-defined function on $\prod_{\nu\notin S}'\widetilde G_\nu$, namely for every $(g_\nu)_{\nu\notin S}\in\prod_{\nu\notin S}'\widetilde G_\nu$, the infinite product $\prod_{\nu\notin S}\varphi_\nu(g_\nu)$ is convergent. Then the function $\prod_{\nu\notin S}\varphi_\nu:\prod_{\nu\notin S}'\widetilde G_\nu\to\BC$ descends to a genuine function $\varphi$ on $\widetilde G_{\BA^S}$. By abuse of notion, the function $\varphi$ is also denoted $\prod_{\nu\notin S}\varphi_\nu$. 
\end{nprop}
And a similar result holds for $\widetilde T_{\BA^S}$. A genuine function on $\widetilde G_{\BA^S}$ or $\widetilde T_{\BA^S}$ is called \emph{factorizable} if it can be obtained from a family $(\varphi_\nu)_{\nu\notin S}$ of local genuine functions in this way. 
\tpoint{Genuine functions on $\widetilde G_\BA$}\label{genuine-global}
Similarly, let $\varphi^S$ be a genuine function on $\widetilde G_{\BA^S}$, let $\varphi_S$ be a genuine function on $\widetilde G_S$. Then we can define a genuine function $\varphi$ on $\widetilde G_\BA=\widetilde G_{\BA^S}\times_{\mu_n(k)}\widetilde G_S$ by descending the function $\varphi^S\cdot\varphi_S$ on $\widetilde G_{\BA^S}\times\widetilde G_S$. By abuse of notation, the genuine function $\varphi$ on $\widetilde G_\BA$ is also denoted $\varphi^S\cdot\varphi_S$. 
\par The following simple lemma will be used without mentioning:
\begin{nlem}
	Let $\varphi^S=\prod_{\nu\notin S}\varphi_\nu$ be a factorizable function on $\widetilde G_{\BA^S}$, let $\varphi_S$ be a function on $\widetilde G_S$, let $\varphi=\varphi^S\cdot\varphi_S$. Then for $g=(g_\nu)_\nu\in\widetilde G_\BA$ and $\eta_1,\eta_2\in T_k$ we have 
	$$\varphi(g\eta_1\eta_2)=\left(\prod_{\nu\notin S}\varphi_\nu(g_\nu \bs_\nu(\eta_1)\bs_\nu(\eta_2))\right)\varphi_S(g_S \bs_S(\eta_1)\bs_S(\eta_2))$$
	where $g_S=(g_\nu)_{\nu\in S}\in\widetilde G_S$. The same is true if $\widetilde G_\BA$ is replaced by $\widetilde T_\BA$.
\end{nlem}
The lemma is a simple consequence of the fact that (\ref{section-global}) is a splitting of $T_k$ in $\widetilde G_\BA$. 
\subsection{Metaplectic Eisenstein series}\label{met-Es}
Still let $k$ be a totally imaginary number field containing and pick $n$ and $S$ to satisfy the conditions in \S \ref{subsub:WMDS-S-conditions}. 

%

\tpoint{$\bG$ over Bad places} We deal with the places in $S$ together. Heuristically, in this section, $k_S$ and $\CO_S$ plays the same role as $\lf$ and $\ro$ in the local case. This idea comes from \cite{BBF:annals}. Let $G_S=\bold G(k_S)$, let $U_S=\prod_{\nu\in S}U_\nu$, $U_S^-=\prod_{\nu\in S}U_\nu^-$, $T_S=\prod_{\nu\in S}T_\nu$, $T_{0,S}=\prod_{\nu\in S}T_{0,\nu}$ be the corresponding subgroups of $G_S$. Let $p_S:\widetilde G_S\to G_S$ be the central extension of $G_S$ by $\mu_n(k)$ defined by a similar pushout as in \S\ref{def:Global-metaplectic-group}. Note that the unipotent radicals $U_S$, $U_S^-$ still splits canonically in $\widetilde G_S$, and the torus $\widetilde T_S=p_S^{-1}(T_S)$ is a central extension of $T_S$ by $\mu_n(k)$. 
\tpoint{The torus $\widetilde T_S$}To distinguish between different places, for $\nu\in\places_k$, $s\in k_\nu$ and $i=1,\cdots,r$ let $\sfh_{i,\nu}(s)$ be the corresponding element in $T_\nu$, where we recall that the notation $\sfh_i$ was defined in \S\ref{Chevalley-torus}. 
\par For every $t=(t_\nu)_{\nu\in S}\in k_S$, let $\mathsf h_{i,S}(t)=\prod_{\nu\in S}\mathsf h_{i,\nu}(t_\nu)$ be the corresponding element in $T_S$, let $\widetilde h_{i,S}(t_\nu)=\prod_{\nu\in S}\widetilde h_{i,\nu}(t_\nu)$ be the corresponding element in $\widetilde T_S$. For $\ul C\in (\fo_S\setminus\{0\})^r$, we define
$$\eta_S(\ul C)=\widetilde h_{1,S}(C_{1})\cdots\widetilde h_{r,S}(C_{r})\in\widetilde T_S.$$ Let $\bs_S:T_S\to\widetilde T_S$ be the map given by 
\begin{equation}\label{section-local}
	\bs_S(\mathsf h_{1,S}(s_{1,\nu})\cdots\mathsf h_{r,S}(s_{r,\nu}))=\widetilde h_{1,S}(s_{1,\nu})\cdots\widetilde h_{r,S}(s_{r,\nu})\text{ for }s_1,\cdots,s_r\in k_S^*.
\end{equation}
This is a section of $p_S:\widetilde T_S\to T_S$ but not a group homomorphism. 
\par Note that $T_{\fo_S}$ is an abelian group by Lemma \ref{def:Omega} (i), and also by this lemma, $\bs_S|_{T_{\fo_S}}:T_{\fo_S}\to\widetilde T_S$ is a group homomorphism, so it a splitting of $T_{\fo_S}\in\widetilde T_S$. We view $T_{\fo_S}$ as a subgroup of $\widetilde T_S$ via this splitting. Moreover, we have
\begin{nprop}
	$\widetilde T_{*,S}=\widetilde T_{0,S}T_{\fo_S}$ is an abelian subgroup of $\widetilde T_S$. 
\end{nprop}
\begin{proof}
	$\widetilde T_{0,S}$ is central in $\widetilde T_{*,S}$ by Lemma \ref{center-of-T}, and $T_{\fo_S}$ is abelian by Lemma \ref{def:Omega}. 
\end{proof}
We expect $\widetilde T_{0,S}$ to be a maximal abelian subgroup of $\widetilde T_S$ so that it plays the role of the maximal abelian subgroup $\widetilde T_{*,\nu}$ of $\widetilde T_\nu$ for $\nu\notin S$, but we don't need this. 
\tpoint{Induced representation of $\widetilde G_S$}The following mimics \S\ref{usual-principal-series}. For an element $a=\mathsf h_{1,S}(C_1)\cdots\mathsf h_{n,S}(C_n)$ of $T_{0,\CO_S}$ and $\mu\in\Lambda\otimes\BC$, we define 
\begin{equation}\label{S-absolute value}
	|a|_S^\mu:= |C_1|_S^{\pair{\mu,\alpha_1^\vee}}\cdots|C_r|_S^{\pair{\mu,\alpha_r^\vee}}
\end{equation}
\par Now every $\lambda\in\Lambda\otimes\BC$ induces a character $\chi_{\lambda,S}:T_{0,S}\to\BC^*$ by $\chi_{\lambda,S}(a)=|a|_S^{\lambda-\rho}$ which is $T_{\fo_S}$-invariant. It extends to a character of $T_{*,S}$ by the composition 
$$T_{*,S}\twoheadrightarrow T_{*,S}/T_{\fo_S}\cong T_{0,S}/T_{0,\fo_S}\xrightarrow{\chi_{\lambda,S}}\BC$$
This character of $T_{*,S}$ is still denoted by $\chi_{\lambda,S}$. Let $\widetilde\chi_{\lambda,S}$ be the genuine character of $\widetilde T_{*,S}$ defined by 
$$\widetilde\chi_{\lambda,S}(\bs_S(a)\zeta)=\epsilon(\zeta)\chi_{\lambda,S}(a)\text{ for }\zeta\in\mu_n(k),\,a\in T_{*,S}.$$
It is right $T_{\fo_S}$-invariant. 
\par Let $i_S(\lambda)=\Ind_{\widetilde T_{*,S}}^{\widetilde T_S}(\widetilde\chi_{\lambda,S})$ be the induced representation. Similar to the previous section, it consists of functions $f:\widetilde T_S\to\BC$ such that $f(aa')=f(a)\widetilde\chi_{\lambda,S}(a')$ for $a\in\widetilde T_S$, $a'\in\widetilde T_{*,S}$. 
\par Let $I_S(\lambda)$ be the induced representation $I_S(\lambda)=\Ind_{\widetilde T_SU_S}^{\widetilde G_S}i_S(\lambda)$ or $\Ind_{\widetilde T_{*,S}U_S}^{\widetilde G_S}(\widetilde\chi_{\lambda,S})$. The second model consists of functions $\Phi:\widetilde G_S\to\BC$ such that 
$$\Phi(gau)=\Phi(g)\widetilde\chi_{\lambda,S}(a)$$
for $g\in\widetilde G_S,\,a\in\widetilde T_{*,S}, u\in U_S$. 
\par Let $\Phi_{\lambda,S}$ be a function in $\Ind_{\widetilde T_{*,S}U_S}^{\widetilde G_S}(\widetilde\chi_{\lambda,S})$. It satisfies
\begin{equation}\label{property-PhiS}
	\Phi_{\lambda,S}(g\bs_S(a)a'\zeta u)=\epsilon(\zeta)\Phi_{\lambda,S}(g)|a|_S^{\lambda-\rho}\text{ for }g\in\widetilde G_S,\,\zeta\in\mu_n(k),\,a\in T_{0,S},\,a'\in T_{\fo_S},\,u\in U_S.
\end{equation}

\spoint{}  Now fix $\lambda\in\Lambda\otimes\BC$. Let $\Phi_{\lambda,\nu}$ be the function $\Phi_\lambda$ defined in \S\ref{usual-principal-series} for $\nu\notin S$. By Proposition \ref{genuine-local-global} we can define a genuine function $\prod_{\nu\notin S}\Phi_{\lambda,\nu}$ on $\widetilde G_{\BA^S}$ because $\Phi_{\lambda,\nu}|_{G_{\ro_\nu}}=1$ for all $\nu\notin S$. We define a genuine function $\Phi_{\lambda,*}:\widetilde G_\BA\to\BC$ by 
\be {} \label{def:Phi_*} \Phi_{\lambda,*}(g)=\left(\prod_{\nu\notin S}\Phi_{\lambda,\nu}\right)\cdot\Phi_{\lambda,S} \ee
in the notation of \S\ref{genuine-global}. 
\begin{nlem}
	The function $\Phi_{\lambda,*}$ is right $T_{0,k}$-invariant. 
\end{nlem}
\begin{proof}
	Let $g=(g_\nu)_\nu\in\widetilde G_\BA$, $a\in T_{0,k}$. We have 
	$$\Phi_{\lambda,*}(ga)=\left(\prod_{\nu\notin S}\Phi_{\lambda,*}(g_\nu\bs_\nu(a))\right)\cdot\Phi_{\lambda,S}(g_S\bs_S(a))=\left(\prod_{\nu\notin S}\Phi_{\lambda,*}(g_\nu)|a|_\nu^{\lambda-\rho}\right)\cdot\Phi_{\lambda,S}(g_S)|a|_S^{\lambda-\rho}=\Phi_{\lambda,*}(g).$$
\end{proof}
\spoint The function $\Phi_{\lambda, *}$ is only invariant under $T_{0,k}$ but not invariant under $T_k$, so following \cite{KP:pmihes} we now introduce 
the function $\Phi_{\lambda,0}(x): \widetilde G_\BA \rr \BC$, 
\be{} \label{def:Phi_0} \Phi_{\lambda,0}(x)=\sum_{\eta\in T_k/T_{0,k}}\Phi_{\lambda,*}(x\eta) \ee
 This function is left $K^S=\prod_{v\notin S} K_v$-invariant and right $U_\BA T_k$ invariant.
 
 \tpoint{Metaplectic Eisenstein series} Using the function $\Phi_{\lambda,0}$, we form the \textit{metaplectic Eisenstein series}
 
\be{} \label{def:met-ES} E(\lambda,\Phi_{\lambda,S},g)=\sum_{\gamma\in G_k/B_k}\Phi_{\lambda,0}( g\gamma) \ee

\begin{nprop}[\cite{MW}]
	\begin{enumerate}[(i)]
		\item Let 
		\begin{equation}\label{Godement-region}
			\Gode=\{\lambda\in\Lambda\otimes\BC:\pair{\Re(\lambda),\alpha_i^\vee}\leq-1\text{ for every simple coroot }\alpha_i^\vee\}.
		\end{equation} be the Godement region. $E(\lambda,\Phi_{\lambda,S},g)$ is absolutely convergent for $(g,\lambda)\in\widetilde G(\BA)\times\Gode$. Moreover, the convergence is uniform on compact subsets of $(g,\lambda)\in\widetilde G(\BA)\times\Gode$, in particular, for a fixed $g\in\widetilde G(\BA)$, the function $E(g,\Phi_{\lambda,S},\lambda)$ is a holomorphic function in $\lambda\in\Gode$. 
		\item For a fixed $g\in\widetilde G(\BA)$, the function $E(g,\Phi_{\lambda,S},\lambda)$ has a meromorphic continuation to $\lambda\in\Lambda\otimes\BC$. 
	\end{enumerate}
\end{nprop}

\tpoint{Functional equations}
\par For $\lambda\in\mathrm{Gode}$, let $M(w,\lambda):\Ind_{\widetilde T_{*,S}U_S}^{\widetilde G_S}(\chi_{\lambda,S})\to \Ind_{\widetilde T_{*,S}U_S}^{\widetilde G_S}(\chi_{w\lambda,S})$ be the intertwiner 
$$M(w,\lambda)\Phi_{\lambda,S}(g)=\int_{U_{w,S}}\Phi_{\lambda,S}(guw)du$$
It is known that the intertwiners $M(w,\lambda)$ also have meromorphic continuations to $\lambda\in\Lambda\otimes\BC$. 
\begin{nprop}[\cite{MW}]
	For every $w\in W$, we have the following functional equation 	\begin{equation}\label{Eis-functional-equation}
		E(g,M(w,\lambda)\Phi_{\lambda,S},w\lambda)=\left(\prod_{\substack{\alpha^\vee\in\Phi_+^\vee\\w^{-1}\alpha^\vee<0}}\frac{\zeta_S(-\pair{\lambda,\alpha^\vee})}{\zeta_S(-\pair{\lambda,\alpha^\vee}+1)}\right)E(g,\Phi_{\lambda,S},\lambda).
	\end{equation}
	where $\zeta_S(z)=\prod_{v\notin S}(1-q_v^{-z})^{-1}$ is the partial zeta function of $k$ with local factors in $S$ removed. 
\end{nprop}

\subsection{Whittaker coefficients}
\tpoint{Remark on Haar measures}In this article we will only do integration on the unipotent radical $U^-$ and its subgroups over local and global fields, so we simply take Tamagawa measures (see \cite{We}). In particular, for any unipotent group $\bold N$ over $k$, the Tamagawa measure on $N_\BA$ induces an invariant measue on the quotient $N_\BA/N_k$, under which the volume of $N_\BA/N_k$ is equal to $1$. 
\tpoint{First Whittaker coefficient}
From now on we will fix the function $\Phi_{\lambda,S}$ and simply denote the Eisenstein series by $E(g,\lambda)$. We take an additive character $\psi:\BA/k\to\BC^*$, which factorizes as $\psi=\prod_\nu\psi_\nu$ for $v\in\places(k)$, where $\psi_\nu:k_\nu\to\BC^*$ is an additive character of the local field $k_\nu$. We suppose that $\psi_\nu$ is non-trivial for every place $\nu$ and unramified for every non-archimedean place $\nu\notin S$. 
\par For $\lambda \in \Gode$ and $g \in \tG_{\ad},$ the first \textit{Whittaker coefficient} of the metaplectic Eisenstein series $E(\lambda,g)$ introduced in \eqref{def:met-ES} is defined to be
\be{} \label{def:Whit-Eis-coeff} W(\lambda, g)=\int_{U^-_\ad/U^-_k}E(\lambda, gu^-)\psi(u^-)^{-1} \, du^-, \ee where $du^-$ is defined to be the (quotient) Haar measure on $U^-_{\ad}/U^-_k$ such that $U^-_{\ad}/U^-_k$ has total volume $1$. Note that since $U^-_\ad/U^-_k$ is compact, the Whittaker coefficient is absolutely convergent for $\lambda\in\Gode$. Our goal in this section is to relate this first Whittaker coefficient to the Weyl group multiple Dirichlet series introduced in \S \ref{subsub:definition-WMDS}.
\tpoint{Statement of main result} We can now state the main result of this section, the so-called \textit{Eisenstein conjecture} of \cite{WMDS1,WMDS2} . The proof will occupy the remainder of this section.

	\begin{nthm}\label{Eis-conj} Let $k$ be a number field containing all $2n$-th roots of unity. Choose the finite set of places $S \subset \places_k$ to satisfy the conditions of \S \ref{subsub:WMDS-S-conditions}. Let $\mf{D}$ be a semisimple simply-connected root datum such that the metaplectic dual root datum $\fD_{(sQ,n)}^\vee$ defined in Lemma \ref{metaplectic-dual-root-data} (ii) is of adjoint type. Let $\lambda\in\Gode$, let $s_i:= \pair{\rho-\lambda,\alpha_i^\vee}=1-\la \lambda, \alphav_i\ra$ for $i \in I.$ Then the coefficient $W(\lambda, 1)$ from \eqref{def:Whit-Eis-coeff} is well-defined and
		 we have  \be{}  \label{main-statmements-Eis-Conj} W(\lambda,1)=[T_{\CO_S}:T_{0,\CO_S}]Z_{\Psi}(s_1,\cdots,s_r), \ee where $Z_{\Psi}(s_1,\cdots,s_r)$ is the Weyl group multiple Dirichlet series attached to $(\mf{D}, n, S, \sQ)$ (see \S \ref{subsub:definition-WMDS} ) with $H(C_1, \ldots, C_r)$ as in \ref{def:H-coefficients} and 
		 \be{} \Psi(C_1,\cdots,C_r):=\BN C_1^{s_1}\cdots\BN C_r^{s_r}\int_{U_S^-}\Phi_{\lambda,S}(u_S^-\widetilde h_{1,S}(C_1)\cdots\widetilde h_{r,S}(C_r))\psi_v(u_S^-)^{-1}du_S^-,\ee
		 where $\psi_S=\prod_{\nu\in S}\psi_\nu$. 
%
	\end{nthm}

The proof will be broken down into three parts: in the rest of this subsection we use a standard unfolding argument to rewrite the Whittaker coefficient as a sum over $\eta\in T_k/T_{0,k}$ of two functions $I^S(\eta)$ and $I_S(\eta)$ that can actually be defined on a larger domain, and we establish the invariance properties of these two functions.
\par Next we rewrite this summation into a summation over $\bigoplus_{v\notin S}\Lambda^\vee/\Lambda_0^\vee$ and compare with the regrouped summation of the WMDS in (\ref{group-sum-Z}), so it suffices to prove the equality for the sub-summations for both the Whittaker coefficient and the WMDS for each $\ul\lambda=(\lambda_\nu^\vee)^\vee\in\bigoplus_{v\notin S}\Lambda^\vee/\Lambda_0^\vee$. In these sub-summations, the $\Psi$ function in the WMDS is equal to the $I_S$-function on the Whittaker side, so it boils down to prove the equality
$$I^S(\ul \lv )=\sum_{\ul C\in \supp (Z;\ul\lambda^\vee)}H(C_1,\cdots,C_r)\BN C_1^{-s_1}\cdots\BN C_r^{-s_r}.$$
\par Then note that $I^S$ factorizes into a product of local integrals $I_\nu$ for $\nu\notin S$, and each $I_\nu$ is equal to the product of the part of the $\nu$-part supported on $-\lambda_\nu^\vee$ times the factor $D(\ul\lambda^\vee;\nu)$ defined in (\ref{D-nu}). Finally we span the product local integrations, the $\nu$-parts glues into the right hand side of the above equality because of Lemma \ref{lem:H-fact}. 
\tpoint{Unfolding the metaplectic Eisenstein series} \label{subsub:unfolding} The first step in our argument will be a standard unfolding argument similar to the one used when computing Whittaker coefficients of Borel Eisenstein series on a reductive group. Note that, due to the local non-uniqueness of Whittaker functionals, this does not result in an Euler product as in the linear case. 
\begin{nclaim} \label{claim:unfolding-W}  For $a \in \tT_{\ad}$ and $\lambda$ as above, we have \be{}   W(\lambda, a):= \sum_{\eta\in T_k/T_{0,k}}\int_{U^-_\BA}\Phi_{\lambda,*}( au^-\eta)\psi(u^-)^{-1}du^-. \ee 
	 \end{nclaim}
\begin{proof}

By definition, for $a \in \tT_{\ad}$, we have 
	\be{} W(\lambda, a)=\int_{U^-_\BA/U^-_k}\sum_{\gamma\in G_k/B_k}\Phi_{\lambda,0}(au^-\gamma)\psi(u^-)^{-1}du^-
.	\ee Decomposing, using the Bruhat decomposition, $G_k/B_k:= \sqcup_{w \in W} \, U_{w, k} \, \dw B_k/B_k$, we find that the above can be decomposed, using the right $B_k$-invariance of $\Phi_{\lambda, 0}$, as
	\be{} \int_{U^-_\BA/U^-_k} \sum_{w\in W}\sum_{u_w \, \in U_{w}}\Phi_{\lambda,0}(au^- \, u_w \dw )\psi(u^-)^{-1}du^- =  \int_{U^-_\BA/U^-_k} \sum_{w\in W}\sum_{\gamma\in U^-_{w,k}}\Phi_{\lambda,0}(au^- \dw \,  u_{-w,k})\psi(u^-)^{-1}du^-
\ee where $U^-_{w, k}= \dw^{-1} \, U_{w, k} \dw$ is a subgroup of $U^-_k$. Writing also $U^{-, w}_k:= U^-_k \cap \dw \, U^-_k \dw^{-1}$, we have a factorization $U^-_k = U^{-, w}_k \, U^-_{w, k}.$ 	A standard Fubini type argument then shows the above can be rewritten as 
	\be{} \sum_{w\in W}\int_{U_\BA^-/U_k^{-, w}}\Phi_{\lambda,0}(au^-w)\psi(u^-)^{-1}du^-. \ee
Similarly decomposing $U^-_{\ad} = U^{-}_{w, \ad} U^{-, w}_{\ad}$ we find that the above is equal to 	
	
\be{} && \sum_{w\in W}\int_{U_{w,\BA}^-}\int_{U_\BA^{w,-}/U_k^{w,-}}\Phi_{\lambda,0}(au^-_1u_2^-w)\psi(u_1^-u_2^-)^{-1}du_1^-du_2^-
	\\&=&\sum_{w\in W}\int_{U_{w,\BA}^-}\left(\int_{U_\BA^{w,-}/U_k^{w,-}}\psi(u_2^-)^{-1}du_2^-\right)\Phi_{\lambda,0}(au^-_1w)\psi(u_1^-)^{-1}du^-. \ee If $w \neq e$, the inner integral over $U^{-, w}_{\ad}/ U^{-, w}_k$ vanishes, so that we are just left with 
\be{} \int_{U^-_\BA}\Phi_{\lambda,0}(au^-)\psi(u^-)^{-1}du^-
.\ee Now, if we substitute in the definition \eqref{def:Phi_0}, the claim follows. \end{proof}

\tpoint{The functions $I^S(\eta), I_S(\eta)$} We now focus on the case when $a=1$. Using the previous Claim, we may now write $W(\lambda, 1)$ as 
\begin{align}\label{first-form-Whittaker-coefficient}
	&\nonumber \sum_{\eta\in T_k/T_{0,k}} \underbrace{ \left(|\eta|_v^{\rho-\lambda} \, \prod_{v\notin S}\int_{U_v^-}\Phi_{\lambda,v}(u_v^-\bs_\nu(\eta))\psi_v(u_v^-)^{-1}du_v^-\right)}_{I^S(\eta)}\, \cdot \, 
	\\&\underbrace{\left(\, |\eta|_S^{\rho-\lambda} \, \int_{U^-(k_S)}\Phi_{\lambda,S}(u_S^-\bs_S(\eta))\psi_v(u_S^-)^{-1}du_S^-\right)}_{I_S(\eta)}.
\end{align}
	Note that both $I^S(\eta)$ and $I_S(\eta)$ are now defined on $T_k$ and right invariant under $T_{0,k}$.  In fact, we may extend these functions by
\begin{equation}\label{extend-I}
	I^S(\eta):=\prod_{\nu\notin S}|\eta_\nu|_\nu^{\rho-\lambda}\int_{U_\nu^-}\Phi_{\lambda,\nu}(u_\nu^-\eta_\nu)\psi_v(u_v^-)^{-1}du_v^-
\end{equation}
for $\eta=\prod_{\nu\notin S}\eta_\nu\in\widetilde T_{\BA^S}$ with $\eta_\nu\in\widetilde T_\nu$, and 
\begin{equation}\label{extend-J}
	I_S(\eta):=|\eta|_S^{\rho-\lambda}\int_{U_S^-}\Phi_{\lambda,S}(u_S^-\eta)\psi_v(u_S^-)^{-1}du_S^-\text{ for }\eta\in\widetilde T_S.
\end{equation}
For every $\nu\notin S$ we also define $I_\nu:\widetilde T_\nu\to\BC$ by 
\begin{equation}\label{I-nu}
	I_\nu(\eta_\nu)=|\eta_\nu|_\nu^{\rho-\lambda}\int_{U_\nu^-}\Phi_{\lambda,\nu}(u_\nu^-\eta_\nu)\psi_\nu(u_\nu^-)^{-1}du_\nu^- 
\end{equation}
Then in the notation introduced after Proposition \ref{genuine-local-global}, we have $I^S=\prod_{\nu\notin S} I_\nu$. Note that here by abuse of notation, we extend the absolute value functions $|\cdot|_\nu: T_\nu\to\BC$ to $\widetilde T_\nu$ via composition with $p_\nu:\widetilde T_\nu\to T_\nu$, and similar for $|\cdot|_S$. 
\begin{nprop}\label{invariance-IJ}
	The functions $I^S,I_S,I_\nu$ are all genuine. The function $I_\nu:\widetilde T_\nu\to\BC$ is right $T_{*,\nu}$-invariant, and the function $I_S:\widetilde T_{S}\to\BC$ is right $T_{*,S}$-invariant. (Note that by Corollary \ref{section-homomorphism-Lambda0}, for every $\nu\notin S$, $\bs_\nu|_{T_{*,\nu}}$ is a splitting of $T_{*,\nu}$, and we view $T_{*,\nu}$ as a subgroup of $\widetilde T_\nu$ via this splitting.)
\end{nprop}
\begin{proof}
	The functions $I_\nu$ and $I_S$ are clearly genuine. By Proposition \ref{genuine-local-global}, the function $I^S=\prod_{\nu\notin S} I_\nu$ is also genuine. 
	\par The invariance of $I_\nu$ follows from the fact that $\Phi_{\lambda,\nu}$ satisfies the property (\ref{property-Phi-lambda}). Similarly, the invariance of $I_S$ follows from the fact that $\Phi_{\lambda,S}$ satisfies the property (\ref{property-PhiS}). 
\end{proof}
\tpoint{}In the end of this subsection we will prove a key property of the functions $I_S$ and $I_\nu$ which links the summation (\ref{first-form-Whittaker-coefficient}) to the summation of Weyl group multiple Dirichlet series. Before stating and proving this key property, we need some preparations. 
\par For any $\ul C \in (\o_S \setminus \{ 0 \})^r$, say $\ul C = (C_1, \ldots, C_r)$ with $C_i \in   (\o_S \setminus \{ 0 \})$, let
\be{} \eta_\nu(\ul C):=\eta_{k_\nu}(\ul C)=\mathsf h_{1,\nu}(C_{1}) \cdots\mathsf h_{r,\nu}(C_{r}) \in  T_{{\nu}} \mbox{ for } \nu \in \places_k, \ee 
Where we recall that the notation $\eta_{k_v}$ was defined in (\ref{eta-field}). 
Let $\bs_\nu:T_\nu\to\widetilde T_\nu$ be the map defined by (\ref{section-local}) for the local field $k_\nu$. 
\begin{nlem}\label{multiplicativity-for-Lambda0}
For $\ul C, \ul C' \in (\o_S ^+)^r$ with $\log_{\omega}( \ul C') \in \Lv_0$ for every place $\omega\notin S$. Then for every place $\nu$ we have \be{} \bs_{\nu} (\eta_\nu(\underline{CC'}) ) = \bs_{\nu}(\eta_\nu(\ul C )) \bs_{\nu}(\eta_\nu(\ul C')) \ee 
\end{nlem}
\begin{proof} 
	If $\nu$ is archimedean there is nothing to prove. If $\nu$ is non-archimedean, suppose $\ul C'=(C_1',\cdots,C_r')$ and $C_i'=\prod_{\omega\notin S}\pi_\omega^{k_i^\omega}$. Then for every $\omega\notin S$ we have $\log_\omega\ul C'=\sum_{i=1}^r k_i^\omega\alpha_i^\vee$ and thus $n_i|k_i^\omega$ for every $\omega\notin S$ and $i=1,\cdots,r$. Thus
	$$\eta_\nu(\ul C')=\left(\prod_{\omega\notin S}\pi_\omega^{k_1^\omega\alpha_1^\vee}\right)\cdots\left(\prod_{\omega\notin S}\pi_\omega^{k_r^\omega\alpha_r^\vee}\right)\in T_{0,\nu}$$
	because $k_i^\omega\alpha_i^\vee\in\Lambda_0^\vee$ for every $\omega\notin S$ and $i=1,\cdots,r$. The lemma follows from Corollary \ref{section-homomorphism-Lambda0}. 
\end{proof}
\spoint Let $C \in \o_S^+$. Recall that we can factor 
\be{} C = \, \prod_{\nu \notin S}  \pi_\nu^{m_{\nu}} \mbox{ with } \mbox{ and almost all of the }  m_{\nu}=0. \ee
For ease of notation, we sometimes just write $C_{\nu} \in \o_S$ for $\pi_\nu^{m_{\nu}}$ and 
\be{} C^{\nu}:= \prod_{\omega \neq \nu, \omega \notin S} C_{\omega}. \ee
  With this notation, we may write \be{} C =  C_{\nu} \cdot C^{\nu} \ee
   A similar factorization holds for $\ul C$ which we then write as 
   \be{} \ul C &=&   (\prod_{\nu \notin S} C_{1, \nu}, \cdots, \prod_{\nu \notin S} C_{r, \nu} ) \mbox{ with } a_i \in \o_S^{\times} \\ &=& ( C_{1, \nu} C_1^{\nu} , \cdots, C_{r, \nu} C_r^{\nu} ).  \ee
   Writing $\ul C_{\nu} := (C_{1, \nu} , \cdots, C_{r, \nu})$ , $\ul C^{\nu}= (C^{\nu}_{1} , \cdots, C^{\nu}_{r})$, we may also write 
   \be{} \ul C =  \, \ul C_{\nu} \ul C^{\nu}, \ee
    where the multiplication is componentwise in the above expression. 
\begin{nlem}\label{i-nu-eta-C}
 Let $\ul C \in (\o_S^+)^r$. For any $\nu \notin S$, we have in $\widetilde T_\nu$ \be{} \bs_\nu \left( \eta_\nu(\ul C) \right) = D(\ul C;\nu)\bs_\nu \left( \eta_\nu( \ul C_{\nu}) \right)  \cdot \bs_\nu \left( \eta_\nu( \ul C^{\nu} ) \right), \ee
where $D(\ul C;\nu)$ is defined by (\ref{D-nu}). 
\end{nlem}
\begin{proof} 	
By Proposition \ref{local-cocycle} we have 
$$\bs_\nu \left( \eta_\nu(\ul C) \right) = d(\ul C_\nu,\ul C^\nu)\bs_\nu \left( \eta_\nu( \ul C_{\nu}) \right)  \cdot \bs_\nu \left( \eta_\nu( \ul C^{\nu} ) \right)$$
The lemma follows from the fact that $D(\ul C;\nu)=d(\ul C_\nu,\ul C^\nu)$, which directly follows from the definitions (\ref{D-nu}) and (\ref{def:d-}). 
\end{proof} 

\tpoint{The key property of the functions $I_S$ and $I_\nu$}

\begin{nlem}\label{lem:IJ-constant-fibers}
	For every $\ul\lambda^\vee=(\lambda_\nu^\vee)_\nu\in\bigoplus_{\nu\notin S}\Lambda^\vee$ we let $\pi^{\ul\lambda^\vee}$ be the element $\prod_\nu\pi_\nu^{\lambda_\nu^\vee}$ in $T_k$. By abuse of notations we also use $\ul\lambda^\vee$ to denote the equivalence class in $\bigoplus_{\nu\notin S}\Lambda^\vee/\Lambda_0^\vee$. Then for $\ul C\in p_Z^{-1}(\ul\lambda^\vee)$ we have 
\begin{equation}
	I_S(\bs_S(\eta_S( \ul C)))=I_S(\bs_S(\pi^{\ul\lambda^\vee})),\,I_\nu(\bs_\nu(\eta_\nu( \ul C)))=I_\nu(\bs_\nu(\pi^{\ul\lambda^\vee}))
\end{equation}
\end{nlem}
\begin{proof}
	Let $\eta'\in T_{0,k}$ be the element such that $\eta_k(\ul C)=\pi^{\ul\lambda^\vee}\eta'$. Then we have $\bs_\nu(\eta_\nu(\ul C))=\bs_\nu(\pi^{\ul\lambda^\vee})\bs_\nu(\eta')$ by Corollary \ref{section-homomorphism-Lambda0}. The second equality then follows from the $\widetilde T_{*,\nu}$-invariance of $I_\nu$. The first equality follows from a similar argument for $S$. 
\end{proof}

\subsection{Proof of Theorem \ref{Eis-conj}}
\tpoint{Step 1: Investigating $W$} Since $\fo_S$ is a unique factorization domain, a set of representatives of $T_k/T_{0,k}$ can be given by $$\{\pi^{\ul\lambda^\vee}u:\ul\lambda^\vee\in\bigoplus_{\nu\notin S}\Lambda^\vee/\Lambda_0^\vee,\,u\in T_{\CO_S}/T_{0,\CO_S}\}.$$
So from (\ref{first-form-Whittaker-coefficient}) we have
\begin{align}\label{second-form-Whittaker-coefficient}
	W(\lambda,1)=&\nonumber\sum_{\eta\in T_k/T_{0,k}}\left(\prod_{\nu\notin S}|\eta|_\nu^{\rho-\lambda}\int_{U_\nu^-}\Phi_{\lambda,\nu}(u_\nu^-\bs_\nu(\eta))\psi_\nu(u_\nu^-)^{-1}du_\nu^-\right)\cdot
	\\\nonumber&\left(|\eta|_S^{\rho-\lambda}\int_{U_S^-}\Phi_{\lambda,S}(u_S^-\bs_S(\eta))\psi_S(u_S^-)^{-1}du_S^-\right)
	\\\nonumber=&\sum_{\ul\lambda^\vee\in\bigoplus_{\nu\notin S}\Lambda^\vee/\Lambda_0^\vee}\sum_{u\in T_{\CO_S}/T_{0,\CO_S}}\left(\prod_{\nu\notin S}|\pi^{\ul\lambda^\vee}|_\nu^{\rho-\lambda}\int_{U_\nu^-}\Phi_{\lambda,\nu}(u_\nu^-\bs_\nu(\pi^{\ul\lambda^\vee})\bs_\nu(u))\psi_\nu(u_\nu^-)^{-1}du_\nu^-\right)\cdot
	\\\nonumber&\left(|\pi^{\ul\lambda^\vee}|_S^{\rho-\lambda}\int_{U_S^-}\Phi_{\lambda,S}(u_S^-\bs_S(\pi^{\ul\lambda^\vee})\bs_S(u))\psi_S(u_S^-)^{-1}du_S^-\right)
	\\=&[T_{\CO_S}:T_{0,\CO_S}]\sum_{\ul\lambda^\vee\in\bigoplus_{v\notin S}\Lambda^\vee/\Lambda_0^\vee}I^S(\bs^S(\pi^{\ul\lambda^\vee}))I_S(\bs_S(\pi^{\ul\lambda^\vee}))
\end{align}
where $I^S$ and $I_S$ are the functions defined in (\ref{extend-I}) and (\ref{extend-J}) respectively. Note that we are slightly abusing notations: for $\ul\lambda^\vee\in \bigoplus_{\nu\notin S}\Lambda^\vee/\Lambda_0^\vee$,  $\pi^{\ul\lambda^\vee}$ is not an element in $T_k$ but an equivalence class in $T_k/T_{0,k}$. However for all the elements in a single equivalence class, the value of the functions $I^S(\bs^S(-))$ and $I_S(\bs_S(-))$ are the same because of their invariance properties in Proposition \ref{invariance-IJ}. 
\tpoint{Step 2: Matching the summations}Comparing the summation in (\ref{second-form-Whittaker-coefficient}) with the regrouping of the summation of the Weyl group multiple Dirichlet series given in (\ref{group-sum-Z}), it suffices to show that
\begin{equation}\label{third-form}
I^S(\bs^S(\pi^{\ul\lambda^\vee}))I_S(\bs_S(\pi^{\ul\lambda^\vee}))=Z_{\ul\lambda^\vee}(s_1,\cdots,s_r)
\end{equation}
for every $\ul\lambda^\vee\in\bigoplus_{v\notin S}\Lambda^\vee/\Lambda_0^\vee$. 
\begin{nprop}\label{Psi-J}
	For every $\ul C=(C_1,\cdots,C_r)\in \supp (Z;\ul\lambda^\vee)$, we have 
	$$\Psi(C_1,\cdots,C_r)=I_S(\bs_S(\pi^{\ul\lambda^\vee}))$$
\end{nprop}
\begin{proof}
Note that we can rewrite the definition of $\Psi$ as
	$$\Psi(\ul C)=|\bs_S(\eta_S(\ul C))|_S^{\rho-\lambda}\int_{U_S^-}\Phi_{\lambda,S}(u_S^-\bs_S(\eta_S(\ul C)))\psi_S(u_S)^{-1}du_S^-=I_S(\bs_S(\eta_S(\ul C))).$$
	The proposition follows from Lemma \ref{lem:IJ-constant-fibers}. 
\end{proof}
\begin{nlem}
	The function $\Psi(C_1,\cdots,C_r)$ satisfies (\ref{property-psi}). 
\end{nlem}  
\begin{proof}
	Since $\Psi(\ul C)=I_S(\bs_S(\eta_S(\ul C)))$, we only need to prove that for $\ul C\in(k_S^\times)^r$ and $\ul C'\in \Omega^r=\fo_S^\times k_S^{\times,n}$, we have 
	$$I_S(\bs_S(\eta_S(\ul C\ul C')))=\left(\prod_{i=1}^r\epsilon(C_i',C_i)_S^{\mathsf Q_i}\right)\left(\prod_{i<j}\epsilon(C_i',C_j)_S^{\mathsf B_{ij}}\right)I_S({\bs_S(\eta_S(\ul C)})).$$
	By Proposition \ref{local-cocycle} we have 
	$$\bs_S(\eta_S(\ul C\ul C'))=\left(\prod_{\nu\in S}d_\nu(\ul C,\ul C')\right)\bs_S(\eta_S(\ul C))\bs_S(\eta_S(\ul C')).$$
	where 
	\begin{equation}\label{d-nu}
		d_\nu(\ul C,\ul C')=\left(\prod_{i=1}^r(C_i,C_i')_\nu^{-\mathsf Q_i}\right)\left(\prod_{1\leq i<j\leq r}(C_i',C_j)_\nu^{\mathsf B_{ij}}\right)
	\end{equation}
	Thus we have 
	\begin{align*}
		&I_S(\bs_S(\eta_S(\ul C\ul C')))=\left(\prod_{\nu\in S}\epsilon(d_\nu(\ul C,\ul C'))\right)I_S(\bs_S(\eta_S(\ul C))\bs_S(\eta_S(\ul C')))
		\\=& \left(\prod_{i=1}^r\epsilon(C_i',C_i)_S^{\mathsf Q_i}\right)\left(\prod_{i<j}\epsilon(C_i',C_j)_S^{\mathsf B_{ij}}\right)I_S(\bs_S(\eta_S(\ul C))\bs_S(\eta_S(\ul C')))
	\end{align*}
	because $I_S$ is a genuine function. Finally, by $\ul C'\in \Omega^r$ we have $\eta_S(\ul C')\in T_{0,S}T_{\fo_S}$, so the lemma follows from the right invariance of $I_S$ under $T_{*,S}=T_{0,S}T_{\fo_S}$. 
\end{proof}
\tpoint{Step 3: gluing local integrations}Comparing (\ref{third-form}) to the $Z_{\ul\lambda^\vee}$ summation (\ref{subsum-Z}) and use Proposition \ref{Psi-J}, it boils down to prove the following
\begin{nlem}\label{lem:I-lambda}  For $\ul \lv \in \oplus_{\nu \notin S} \, \Lv/ \Lv_0$, we have 
\begin{equation} \label{I:Z-H} 
	I^S(\ul \lv )=\sum_{\ul C\in \supp (Z;\ul\lambda^\vee)}H(C_1,\cdots,C_r)\BN C_1^{-s_1}\cdots\BN C_r^{-s_r}.
\end{equation}
\end{nlem} 
\begin{proof} For $\ul C=(C_1,\cdots,C_r)\in \supp (Z;\ul\lambda^\vee)$, first we have $I^S(\pi^{\ul\lambda^\vee})=\prod_{\nu\notin S}I_\nu(\bs_\nu\eta_\nu( \ul C ) ).$ From Lemma \ref{i-nu-eta-C} we have 
\be{} \bs_\nu( \eta_\nu( \ul C )) =D(\ul C;\nu) \, \bs_{\nu} ( \eta_\nu( \ul C_\nu)) \cdot \bs_\nu( \eta_\nu(\ul C^{\nu})).\ee 
Hence 
\be{} I_v (\bs_k \, \eta_\nu(\ul C)) = I_\nu (D(\ul C;\nu) \, \bs_{\nu} ( \eta_\nu ( \ul C_\nu )) \cdot \bs_\nu( \eta_\nu(\ul C^{\nu}))) 
	= \epsilon(D(\ul C;\nu))  I_\nu ( \bs_{\nu} ( \eta_\nu( \ul C_\nu )). \ee 
	Thus we have 
	\be{} I(\bs_\nu \eta_\nu(\ul C)) = \underbrace{\prod_{ \nu \notin S} \, \epsilon(D(\ul C,\nu))}_{\epsilon(D(\ul C))}  \cdot \prod_{\nu \notin S } I_{\nu}(\bs_{\nu} \eta_\nu(\ul C_\nu)). \ee
Write $\lv_{\nu}:= \log_\nu \, \eta_\nu(\ul C)$, so that $\ul \lambda^\vee = (\lv_{\nu})_{\nu}$. By Corollary \ref{Whittaker-H}, 
\be{} I_{\nu} (\bs_{\nu} \eta_\nu( \ul C_\nu) )=I_\nu(\bs_\nu(\pi_\nu^{\lambda^\vee_\nu})) = \sum_{\substack{k_1,\cdots,k_r\geq0\\k_1\alphav_1+\cdots+k_r\alphav_r-\lv_\nu \in\Lambda_0^\vee}}H(\pi_\nu^{k_1},\cdots,\pi_\nu^{k_r})q_\nu^{-k_1s_1-\cdots-k_rs_r} . \ee
We are done using Lemma \ref{lem:H-fact}. 

\end{proof}

\appendix
\section{Intertwiners and Chinta-Gunnells actions}
In this appendix we prove the following formula: 
\begin{nthm}
	Let $s_i$ be a simple reflection, let $\varphi_{\xi^\vee}\in M_\univ$ be a function such that $\CW(\varphi_{\xi^\vee})=e^{\xi^\vee}$ (such a function exists due to Proposition \ref{surjectivity-Whittaker}) Then
	$$W(I_{s_i}\varphi_{\xi^\vee})=\frac{1-q^{-1}}{1-e^{n(\alpha_i^\vee)\alpha_i^\vee}}e^{s_i\xi^\vee+\res_{n(\alpha_i^\vee)}(\pair{\xi^\vee,\alpha_i})\alpha_i^\vee}+q^{-1}\bold g_{(1+\pair{\xi^\vee,\alpha_i})\mathsf Q(\alpha_i^\vee)}e^{s_i\bullet\xi^\vee}$$
	where we note that the intertwiner $I_{s_i}$ is defined in Theorem \ref{Intertwiner}. 
\end{nthm}
\begin{proof}
	The condition that $\CW(\varphi_{\xi^\vee})=e^{\xi^\vee}$ is equivalent to the condition that 
	$$\int_{U^-}\varphi_{\xi^\vee}(u^-\pi^{\lambda^\vee})\psi(u^-)^{-1}du^-=\begin{cases}
		q^{\pair{\rho,\lambda^\vee}}e^{\xi^\vee+\lambda^\vee}, & \xi^\vee+\lambda^\vee\in\Lambda_0^\vee\\
		0, & \text{otherwise}.
	\end{cases}$$
Then
	\begin{align*}
		&\CW(I_{s_i}\varphi_{\xi^\vee})=\sum_{\lambda^\vee\in\Lambda^\vee/\Lambda_0^\vee}\left(q^{-\pair{\rho,\lambda^\vee}}\int_{U^-}(I_{s_i}\varphi_{\xi^\vee})(u^-\pi^{\lambda^\vee})\psi(u^-)^{-1}du^-\right)e^{-\lambda^\vee}
		\\=&\sum_{\lambda^\vee\in\Lambda^\vee/\Lambda_0^\vee}\left(q^{-\pair{\rho,\lambda^\vee}}s_i\int_{U^-}\int_F\varphi_{\xi^\vee}(u^-\pi^{\lambda^\vee}x_i(s)s_i)\psi(u^-)^{-1}du^-ds\right)e^{-\lambda^\vee}
		\\=&\sum_{\lambda^\vee\in\Lambda^\vee/\Lambda_0^\vee}\left(q^{-\pair{\rho,\lambda^\vee}}s_i\int_{U^-}\int_F\varphi_{\xi^\vee}(u^-\pi^{\lambda^\vee}x_{-i}(s^{-1})(-s)^{\alpha_i^\vee}x_i(-s^{-1}))\psi(u^-)^{-1}du^-ds\right)e^{-\lambda^\vee}
		\\=&\sum_{\lambda^\vee\in\Lambda^\vee/\Lambda_0^\vee}\left(q^{-\pair{\rho,\lambda^\vee}}s_i\int_{U^-}\int_F\varphi_{\xi^\vee}(u^-x_{-i}(\pi^{-\pair{\lambda^\vee,\alpha_i}}s^{-1})\pi^{\lambda^\vee}s^{\alpha_i^\vee})\psi(u^-)^{-1}du^-ds\right)e^{-\lambda^\vee}
		\\=&\sum_{\lambda^\vee\in\Lambda^\vee/\Lambda_0^\vee}\left(q^{-\pair{\rho,\lambda^\vee}}s_i\int_{U^-}\int_F\varphi_{\xi^\vee}(u^-\pi^{\lambda^\vee}s^{\alpha_i^\vee})\psi(\pi^{-\pair{\lambda^\vee,\alpha_i}}s^{-1})\psi(u^-)^{-1}du^-ds\right)e^{-\lambda^\vee}
	\end{align*}
	Next we apply $F-\{0\}=\bigsqcup_{k\in\BZ}\pi^{k}\CO^\times$ and write $s=\pi^kr$. Note that $s^{\alpha_i^\vee}=h_i(s)=h_i(\pi^{k}r)=h_i(\pi^{k})h_i(r)(\pi^{k},r)^{-\mathsf Q(\alpha_i^\vee)}=\pi^{k\alpha_i^{\vee}}r^{\alpha_i^\vee}(r,\pi)^{k\sQ_i}$. So the above is equal to 
	{\footnotesize\begin{align*}
		&\sum_{\lambda^\vee\in\Lambda^\vee/\Lambda_0^\vee}\left(q^{-\pair{\rho,\lambda^\vee}}s_i\int_{U^-}\sum_{k\in\BZ}q^{-k}\int_{\CO^\times}(r,\pi)^{k\mathsf Q_i}\varphi_{\xi^\vee}(u^-\pi^{\lambda^\vee}\pi^{k\alpha_i^\vee}r^{\alpha_i^\vee})\psi(\pi^{-\pair{\lambda^\vee,\alpha_i}-k}r^{-1})\psi(u^-)^{-1}du^-dr\right)e^{-\lambda^\vee}
		\\=&\sum_{\lambda^\vee\in\Lambda^\vee/\Lambda_0^\vee}\left(q^{-\pair{\rho,\lambda^\vee}}s_i\int_{U^-}\sum_{k\in\BZ}q^{-k}\int_{\CO^\times}(r,\pi)^{k\mathsf Q_i}\varphi_{\xi^\vee}(u^-\pi^{\lambda^\vee+k\alpha_i^\vee})\psi(\pi^{-\pair{\lambda^\vee,\alpha_i}-k}r^{-1})\psi(u^-)^{-1}du^-dr\right)e^{-\lambda^\vee}
		\\=&\sum_{\lambda^\vee\in\Lambda^\vee/\Lambda_0^\vee}\left(\sum_{k\in\BZ}q^{-\pair{\rho,\lambda^\vee}-k}\left(\int_{\CO^\times}(r,\pi)^{k\mathsf Q_i}\psi(\pi^{-\pair{\lambda^\vee,\alpha_i}-k}r^{-1})dr\right)\left(s_i\int_{U^-}\varphi_{\xi^\vee}(u^-\pi^{\lambda^\vee+k\alpha_i^\vee})\psi(u^-)^{-1}du^-\right)\right)e^{-\lambda^\vee}
	\end{align*}}
	We first look at the integration $$\int_{U^-}\varphi_{\xi^\vee}(u^-\pi^{\lambda^\vee+k\alpha_i^\vee})\psi(u^-)^{-1}du^-.$$ It vanishes if $\lambda^\vee+k\alpha_i^\vee+\xi^\vee$ is not in $\Lambda_0^\vee$. If there exists $k$ such that $\lambda^\vee+k\alpha_i^\vee+\xi^\vee\in\Lambda_0^\vee$, since the summation over $\lambda^\vee$ is over a set of representatives for the quotient $\Lambda^\vee/\Lambda_0^\vee$, we have the flexibility to adjust $\lambda^\vee$ in the same $\Lambda_0^\vee$-coset, so we can simply take $\lambda^\vee=-\xi^\vee-k\alpha_i^\vee$. This means that we can manage the summation over $\lambda^\vee$ to be a summation over $-\xi^\vee-m\alpha_i^\vee$ for $m$ running through a complete residue system modulo $n(\alpha_i^\vee)$. Here we still have the flexibility to adjust each $m$ by a multiple of $n(\alpha_i^\vee)$. 
	\par So from now on we set $\lambda^\vee=-\xi^\vee-m\alpha_i^\vee$ and sum over $m\mod n(\alpha_i^\vee)$. Then the summation over $k$ is actually a summation over $k\in\BZ$ with the same residue modulo $n(\alpha_i^\vee)$ as $m$ (such that $\lambda^\vee+k\alpha_i^\vee+\xi^\vee=(k-m)\alpha_i^\vee\in\Lambda_0^\vee$).
	\par Next we look at the integration
	$$\int_{\CO^\times}(r,\pi)^{k\mathsf Q_i}\psi(\pi^{-\pair{\lambda^\vee,\alpha_i}-k}r^{-1})dr$$
	 which is essentially a Gauss sum defined in \S \ref{Gauss-sums}. In particular by the discussion in \S \ref{Gauss-sums}, it vanishes if $-\pair{\lambda^\vee,\alpha_i}-k<-1$, so for $k$ it suffices to sum over $k\leq-\pair{\lambda^\vee,\alpha_i}+1$ and $k\equiv m\pmod{n(\alpha_i^\vee)}$. The boundary point $k=-\pair{\lambda^\vee,\alpha_i}+1$ is included in the summation only if $-\pair{\lambda^\vee,\alpha_i}+1\equiv\pair{\xi^\vee,\alpha_i}+2m+1\equiv m\pmod{n(\alpha_i^\vee)}$, namely $m\equiv-\pair{\xi^\vee,\alpha_i}-1\pmod{n(\alpha_i^\vee)}$. Since we have the flexibility to modify $m$ by a multiple of $n(\alpha_i^\vee)$ we can simply take $m=-\pair{\xi^\vee,\alpha_i}-1$, then $k=-\pair{\lambda^\vee,\alpha_i}+1=\pair{\xi^\vee,\alpha_i}+2m+1=-\pair{\xi^\vee,\alpha_i}-1=m$. The contribution to the total sum is then equal to 
	\begin{align*}
		&q^{-\pair{\rho,-\xi^\vee-m\alpha_i^\vee}-k}q^{-1}\bold g_{(\pair{\xi^\vee,\alpha_i}+1)\mathsf Q(\alpha_i^\vee)}(s_iq^{\pair{\rho,-\xi^\vee}})e^{\xi^\vee-(\pair{\xi^\vee,\alpha_i}+1)\alpha_i^\vee}
		\\=&q^{-1}\bold g_{(\pair{\xi^\vee,\alpha_i}+1)\mathsf Q(\alpha_i^\vee)}e^{s_i\bullet\xi^\vee}
	\end{align*}
	\par For other $m$, the boundary point is not included and the summation on $k$ is over $k\leq-\pair{\lambda^\vee,\alpha_i}$. In this case, also as discussed in \S \ref{Gauss-sums}, the Gauss sum does not vanish only if $n|k\mathsf Q(\alpha_i^\vee)$, namely $n(\alpha_i^\vee)|k$, and in this case the Gauss sum integration is equal to $(1-q^{-1})$. However, our summation of $k$ is also restricted to those with the same residue as $m$ modulo $n(\alpha_i^\vee)$. This means that the only other congruence class with a nonzero contribution to the sum is the congruence class of $0$, and by our flexibility on choosing $m$ we simply take $m=0$. Then the summation of $k$ is over $k\leq\pair{\xi^\vee,\alpha_i}$ and $n(\alpha_i^\vee)|k$, namely the summation is over 
	$$\{k=\pair{\xi^\vee,\alpha_i}-\res_{n(\alpha_i^\vee)}(\pair{\xi^\vee,\alpha_i})-jn(\alpha_i^\vee):j\geq0\}$$
	and the contribution is equal to 
		$$$$
		\begin{align*}
		&\sum_{j\geq0}q^{\pair{\rho,\xi^\vee}-k}(1-q^{-1})(s_iq^{\pair{\rho,-\xi^\vee+k\alpha_i^\vee}}e^{k\alpha_i^\vee})e^{\xi^\vee}
		\\=&\sum_{j\geq0}(1-q^{-1})e^{\xi^\vee+(-\pair{\xi^\vee,\alpha_i}+\res_{n(\alpha_i^\vee)}(\pair{\xi^\vee,\alpha_i})+jn(\alpha_i^\vee))\alpha_i^\vee}
		\\=&\frac{1-q^{-1}}{1-e^{n(\alpha_i^\vee)\alpha_i^\vee}}e^{s_i\xi^\vee+\res_{n(\alpha_i^\vee)}(\pair{\xi^\vee,\alpha_i})\alpha_i^\vee}
	\end{align*}
	Hence $W(I_{s_i}\varphi_{\xi^\vee})$ is equal to the sum of the two contributions above (note that this is true even when $-\pair{\xi^\vee,\alpha_i}-1\equiv0\pmod{n(\alpha_i^\vee)}$: in this case the two contributions are contributions over different $k$'s). 
\end{proof}
\section{Factorizable functions on the torus and twisted multiplicativity}
\tpoint{Twisted multiplicativity as a kind of factorizability}
In this appendix we record a simple result relating factorizable functions on the torus and twisted multiplicative coefficients. We keep the same notations as in Section 5. 
\begin{nprop}\label{factorization-twisted-mult}
	For every $\nu\notin S$, let $f_\nu: \widetilde T_\nu/T_{\ro_\nu}\to \BC$ be a right $T_{\ro_\nu}$-invariant genuine function such that $f_\nu(1)=1$. Let $f^S=\prod_{\nu\notin S}f_\nu$ be the genuine function on $\widetilde T_{\BA^S}$ as in Proposition \ref{genuine-local-global}. For $\ul C\in(\fo_S^+)^r$ we define a function
	$$H_f(\ul C)=f^S(\bs^S(\eta^S(\ul C)))=\prod_{\nu\notin S}f_\nu(\bs_\nu(\eta_\nu(\ul C))).$$
	Then the function is twisted multiplicative, namely it satisfies (\ref{twisted-multiplicativity}). 
\end{nprop}
\begin{proof}
	For $\ul C$ and $\ul C'$ coprime, let $\Sigma\subseteq\places_k$ be the set of prime factors of $C$ and let $\Sigma'\subseteq\places_k$ be the set of prime factors of $C'$, then $\Sigma\cap\Sigma'=\emptyset$. For $\omega\notin S$ such that $\omega\notin\Sigma\sqcup\Sigma'$ we have $\eta_\omega(\ul C\ul C')\in T_{\ro_\omega}$, so $f_\omega(\bs_\omega(\eta_\omega(\ul C\ul C')))=1$. This implies that 
	\begin{align*}
		&H_f(\ul C\ul C')=\prod_{\nu\in\Sigma\cup\Sigma'}f_\nu(\bs_\nu(\eta_\nu(\ul C\ul C')))
		\\=&\left(\prod_{\nu\in\Sigma}f_\nu(\bs_\nu(\eta_\nu(\ul C\ul C')))\right)\left(\prod_{\nu\in\Sigma'}f_\nu(\bs_\nu(\eta_\nu(\ul C'\ul C)))\right)
		\\=&\left(\prod_{\nu\in\Sigma}f_\nu(\bs_\nu(\eta_\nu(\ul C))\bs_\nu(\eta_\nu(\ul C')))\epsilon(d_\nu(\ul C,\ul C'))\right)\left(\prod_{\nu\in\Sigma'}f_\nu(\bs_\nu(\eta_\nu(\ul C'))\bs_\nu(\eta_\nu(\ul C)))\epsilon(d_\nu(\ul C',\ul C))\right)
		\\=&\left(\prod_{\nu\in\Sigma}f_\nu(\bs_\nu(\eta_\nu(\ul C')))\epsilon(d_\nu(\ul C,\ul C'))\right)\left(\prod_{\nu\in\Sigma'}f_\nu(\bs_\nu(\eta_\nu(\ul C)))\epsilon(d_\nu(\ul C',\ul C))\right)
		\\=&H_f(\ul C)H_f(\ul C')\epsilon\left(\prod_{\nu\in\Sigma}d_\nu(\ul C,\ul C')\right)\epsilon\left(\prod_{\nu\in\Sigma'}d_\nu(\ul C',\ul C)\right)
	\end{align*}
	where we used Lemma \ref{local-cocycle}, (\ref{d-nu}) and the properties of $f_\nu$. Finally we note that by (\ref{d-nu}) and (\ref{power-residue-symbol}) we have \begin{align*}
		&\prod_{\nu\in\Sigma}d_\nu(\ul C,\ul C')=\prod_{\nu\in\Sigma}\left(\prod_{i=1}^r(C_i,C_i')_\nu^{-\mathsf Q_i}\right)\left(\prod_{1\leq i<j\leq r}(C_i',C_j)_\nu^{\mathsf B_{ij}}\right)
		\\=&\prod_{i=1}^r\leg{C_i'}{C_i}_S^{\sQ_i}\prod_{1\leq i<j\leq r}\leg{C_i'}{C_j}_S^{\sB_{ij}}.
	\end{align*}
	and similarly 
	$$\prod_{\nu\in\Sigma'}d_\nu(\ul C',\ul C)=\prod_{i=1}^r\leg{C_i}{C_i'}_S^{\sQ_i}\prod_{1\leq i<j\leq r}\leg{C_i}{C_j'}_S^{\sB_{ij}}$$
	Thus the function $H_f$ satisfies (\ref{twisted-multiplicativity}) on $(\fo_S^+)^r$. 
\end{proof}
In view of this Proposition, twisted multiplicativity should be thought as a metaplectic version of factorizability. Namely, $\widetilde T_{\BA^S}=\prod_{\nu\notin S}\widetilde T_\nu$ is a central extension of $T_{\BA^S}$ by $\mu_n(k)$, and there is a natural notion of factorizable genuine functions on $\widetilde T_{\BA^S}$. Now we have a injective natural group homomorphism $(\fo_S^+)^r\to T_{\BA^S}$ by $\ul C\mapsto(\eta_\nu(\ul C))_{\nu\notin S}$. We can thus pullback the central extension $\widetilde T_{\BA^S}\to T_{\BA^S}$ to $(\fo_S^+)^r$, which yields a central extension of the group $(\fo_S^+)^r$ denoted $\widetilde{(\fo_S^+)^r}$, and the section $\bs^S=\prod_{\nu\notin S}\bs_\nu$ induces a section $\bs$ of $\widetilde{(\fo_S^+)^r}\to (\fo_S^+)^r$. 
\begin{center}
\begin{tikzcd}
\widetilde{(\fo_S^+)^r} \arrow[r] \arrow[d, "p"]        & \widetilde T_{\BA^S} \arrow[d, "p"']      \\
(\fo_S^+)^r \arrow[r, hook] \arrow[u, "\bs", bend left] & T_{\BA^S} \arrow[u, "\bs^S"', bend right]
\end{tikzcd}
\end{center}
By the section $\bs$, a genuine function on $\widetilde{(\fo_S^+)^r}$ pulls back to a function on $(\fo_S^+)^r$. The above Proposition then asserts that if we pullback a \emph{factorizable} function on $\widetilde T_{\BA^S}$ to $\widetilde{(\fo_S^+)^r}$ and then to $(\fo_S^+)^r$, we results in a twisted multiplicative function on $(\fo_S^+)^r$. 
\tpoint{A question}Unfortunately, Proposition \ref{factorization-twisted-mult} does not give directly a proof of the Eisenstein conjecture, because the functions $I_\nu$ does not satisfy $I_\nu(1)=1$. Also, the value of $I^S(\eta^S(\ul C))$ is not equal to $H(\ul C)$ unless we make an additional assumption that $n$ is larger than \emph{twice} of the dual Coxeter number of $\bG$. So what we did in the proof of the Eisenstein conjecture is more complicated. 
\par However, in view of Proposition \ref{factorization-twisted-mult}, it is tempting to ask the following question: 
\begin{nqe}
	Is there a genuine factorizable function $f^S$ on $\widetilde T_{\BA^S}$ such that $H(\ul C)=H_f(\ul C)$? 
\end{nqe}
\bibliographystyle{abbrv}
\bibliography{MetaplecticEisensteinSeries.bib}

\end{document}